\documentclass[11pt,twoside,reqno]{amsart}

\usepackage{amsmath,amsfonts,amsthm,amssymb}
\usepackage{epstopdf}
\usepackage{graphicx,psfrag,overpic}
\usepackage{subfigure}
\usepackage{caption2}
\usepackage{cite}
\usepackage{bm,enumerate}
\usepackage{cases}
\usepackage[all,cmtip,line]{xy}
\usepackage{array,CJK}
\usepackage{color}
\usepackage{booktabs}
\usepackage{underscore}
\usepackage{color}
\usepackage{appendix}
\usepackage{makecell}
\usepackage{pgfplots}
\usepackage{tikz}
\usepackage{pstricks,pst-node}
\usepackage{epstopdf}
\usepackage{cleveref}

\numberwithin{equation}{section}

\usepackage{graphics}
\usepackage{epsfig}
\newtheoremstyle{MyThmStyleb}
{}
{}
{}
{}
{\bfseries}
{}
{ }
{\thmname{#1\thmnumber{ #2\hspace{0.5em}}}\thmnote{(#3)}}

\newtheorem{claim}{\bf \t}[part]

\theoremstyle{plain}
\newtheorem{theorem}{Theorem}[section]

\newtheorem{proposition}[theorem]{Proposition}

\theoremstyle{definition}
\newtheorem{example}{Example}

\newtheorem{definition}{Definition}[section]

\theoremstyle{MyThmStyleb}
\newtheorem{test}{Test}

\theoremstyle{remark}
\newtheorem{remark}{Remark}[section]

\def\ep{\epsilon}
\def\pt{\partial}

\def\br{\bar{\rho}}
\def\bu{\bar{u}}

\def\bE{\bar{E}}

\def\bff{\mathbf{f}}
\def\bg{\mathbf{g}}
\def\bw{\mathbf{w}}

\def\bB{\mathbf{B}}
\def\bC{\mathbf{C}}

\def\bF{\mathbf{F}}
\def\bD{\mathbf{D}}
\def\bA{\mathbf{A}}
\def\bI{\mathbf{I}}
\def\bM{\mathbf{M}}

\def\bG{\mathbf{G}}
\def\bH{\mathbf{H}}

\def\bfm{\mathbf{m}}
\def\bfq{\mathbf{q}}

\def\bfv{\mathbf{v}}
\def\bfu{\mathbf{u}}

\def\by{\mathbf{y}}
\def\bfU{\mathbf{U}}

\def\tbB{\widetilde{\mathbf{B}}}
\def\fE{\mathcal{E}}

\def\tbw{\widetilde{\mathbf{w}}}
\def\diag{\mathrm{diag}}
\def\sign{\mathrm{sign}}

\def\wbB{\widetilde{\mathbf{B}}}
\def\wbC{\widetilde{\mathbf{C}}}
\def\wbD{\widetilde{\mathbf{D}}}

\def\tw{\widetilde{w}}
\def\bh{\mathbf{h}}

\def\fL{\mathcal{L}}
\def\bz{\mathbf{z}}
\def\tB{\widetilde{B}}

\newcommand \R{\mathbb{R}}

\theoremstyle{plain}
\begin{document}

\title[]
{Relaxation schemes for entropy dissipative system of viscous conservation laws}
\author{Tuowei Chen}
\address{Tuowei Chen, Institute of Applied Physics and Computational Mathematics, 100088 Beijing, P. R. China;}
\email{tuowei_chen@163.com}
\author{Jiequan Li}
\address{Jiequan Li, Academy for Multidisciplinary Studies, Capital Normal University, 100048 Beijing, P. R. China, and State Key Laboratory for Turbulence Research and Complex System, Peking University, 100871 Beijing, P. R. China;}
\email{jiequan@cnu.edu.cn}

\date{\today}

\begin{abstract}
In this paper, we introduce a hyperbolic model for entropy dissipative system of viscous conservation laws via a flux relaxation approach. We develop numerical schemes for the resulting hyperbolic relaxation system by employing the finite-volume methodology used in the community of hyperbolic conservation laws, e.g., the generalized Riemann problem method. For fully discrete schemes for the relaxation system of scalar viscous conservation laws, we show the asymptotic preserving property in the coarse regime without resolving the relaxation scale and prove the dissipation property by using the modified equation approach. Further, we extend the idea to the compressible Navier-Stokes equations. Finally, we display the performance of our relaxation schemes by a number of numerical experiments. 
\end{abstract}
\keywords{Relaxation schemes, viscous conservation laws, entropy, Chapman-Enskog expansion, asymptotic preserving property, GRP solver}
\subjclass[2020]{
	49M20, 65M08, 76M12, 35Q30, 76N15
}

\maketitle

\section{Introduction}
In this paper, we study a class of multi-dimensional viscous conservation laws that are entropy dissipative, including the Navier-Stokes equations. The basic idea is to employ the flux relaxation approach. For any given entropy dissipative system of viscous conservation laws, we construct a hyperbolic system with stiff relaxation terms that approximates the original system such that we can employ the finite-volume (FV) methodology developed for conservation laws. 

In the past decades, relaxation methods have gained significant popularity as effective tools in designing numerical schemes for gas dynamics. For example, Jin and Xin \cite{JX1995} proposed the relaxation schemes for general systems of conservation laws via the flux relaxation approach. Xu \cite{Xu2001} developed the gas-kinetic scheme (GKS) for the Navier-Stokes equations based on the Bhatnagar-Gross-Krook (BGK) kinetic model which includes a relaxation collision term reflecting the viscosity and heat conduction effects. Nishikawa \cite{Nishikawa2011} constructed a hyperbolic Navier-Stokes model with relaxation for the diffusive flux and investigated numerical methodologies for computing steady-state solutions. 

Recently, based on the Lax-Wendroff (LW) methodology \cite{LW1960} in the FV framework, temporal-spatial coupling numerical methods exhibit compact, robust and efficient for the computation of compressible fluid flows \cite{Li2019,LD2016,PXLL2016,LL2023}. However, due to the change of governing equations from hyperbolic to hyperbolic-parabolic type, it is still a challenging task to extend the FV framework for hyperbolic conservation laws to solving the Navier-Stokes equations by applying LW type Euler solvers, such as the generalized Riemann problem (GRP) solver \cite{BF2003,BL2007,BLW2006}.

The idea of this paper can be illuminated by considering the Cauchy problem for the 1-D scalar viscous conservation laws
\begin{equation}
	u_t+f(u)_x=(\mu\phi(u) u_{x})_x, \,\,\,\,\, (x,t)\in \R\times\R_+, \label{sc}
\end{equation} 
with initial data
\begin{equation}
	u(x,0)=u_0(x).     \label{scini}
\end{equation}
Here, $u\in \mathcal{D}\subset\R$ with $\mathcal{D}$ denoting the state domain, $f(u)$ and $\phi(u)\geq 0$ are smooth functions of $u$, and $\mu$ is a positive constant. Note that \eqref{sc} has an entropy dissipative nature: for any function $\eta(u)$ with $\eta^{\prime\prime}(u)>0$, \eqref{sc} is endowed with an entropy-entropy flux pair $\big(\eta(u),g(u)\big)$, where $g^{\prime}(u)=\eta^\prime(u)f^\prime(u)$, such that for smooth solutions of \eqref{sc}, the entropy inequality holds, 
\begin{equation}
	\eta(u)_t+g(u)_x-(\eta^\prime(u)\mu\phi(u)u_x)_x=-\eta^{\prime\prime}(u)\mu\phi(u)u^2_x\leq 0.
\end{equation}

Inspired by the Jin-Xin method \cite{JX1995} and the GKS method \cite{Xu2001}, we introduce a hyperbolic system for \eqref{sc}--\eqref{scini} via relaxation for both the convective fluxes and the diffusive fluxes:
\begin{equation}
	\left.\begin{cases}
		u_t+v_x=0,\\
		v_t+(\frac{\mu\phi(u)}{\epsilon}+a^2)u_x=\frac{f(u)-v}{\epsilon},
	\end{cases}\label{re}
	\right.
\end{equation}
with the initial data
\begin{equation}
	u(x,0)=u_0(x),\,\,\,\,
	v(x,0)=v_0(x):= f(u_0(x))-\mu(u_{0}(x)) u^{\prime}_{0}(x),\label{initial}
\end{equation}
where $v\in\R$ is an artificial variable, $\epsilon$ is a small positive parameter standing for the relaxation time, and $a$ is a constant that will be chosen later. Hereafter we call \eqref{re}--\eqref{initial} the relaxation system. Similar to the Jin-Xin relaxation system for conservation laws \cite{JX1995}, in the relaxation limit $(\epsilon \to 0+)$, the relaxation system \eqref{re} can be approximated to leading order by
\begin{equation}
	\left.\begin{cases}
		v=f(u)-\mu\phi(u) u_x,\\
		u_t+f(u)_x-(\mu\phi(u) u_{x})_x=0.
	\end{cases}\right.\label{re1}
\end{equation}
The state satisfying $\eqref{re1}_1$ is called the local equilibrium. $\eqref{re1}_2$ is exactly the original viscous conservation law \eqref{sc}. For smooth solutions, the first order approximation to \eqref{re} reads
\begin{equation}
	\begin{aligned}
		u_t+f(u)_x=&\left((\mu\phi(u)+\ep(a^2-f^\prime(u)^2))u_{x}\right)_x\\
		&+\epsilon\Big(f^\prime(u)(\mu\phi(u)u_{x})_x
		+\big(\mu\phi(u)(f(u)-\mu\phi(u)u_{x})_x\big)_{x}\Big)_x,
	\end{aligned}\label{re2}
\end{equation}
which is derived by using the Chapman-Enskog expansion method \cite{CE1990} and the detailed derivation of \eqref{re2} will be shown later in \cref{Leading} for a general case. It is expected that \eqref{re2} should preserve the entropy dissipative nature of the original equation \eqref{sc}. Therefore, to ensure the dissipation property of \eqref{re2}, we require that the coefficient $\mu\phi(u)+\ep(a^2-f^\prime(u)^2)\geq0$, meaning that $a$ should be chosen such that
\begin{equation}
	-\sqrt{\mu\phi(u)/\epsilon+a^2}\leq f^\prime(u)\leq \sqrt{\mu\phi(u)/\epsilon+a^2},\,\,\,\,\,\,\forall\, u\in\mathcal{D}.    \label{1Dsc3}
\end{equation}
Then, it is expected that appropriate numerical discretizations (hereafter called the relaxation scheme) to the relaxation system \eqref{re} yield accurate approximations of the original equation \eqref{sc} when $\epsilon$ is sufficiently small. We will develop relaxation schemes for viscous conservation laws in the FV framework and then extend the idea to the compressible Navier-Stokes equations. As an example, an implicit-explicit GRP (IMEX-GRP) method will be proposed to overcome the stiffness of the source term.
\begin{remark}
	For the inviscid case that $\mu = 0$, the condition \eqref{1Dsc3} is referred to as the subcharacteristic condition \cite{Liu1987}, i.e. $|f^\prime(u)|\leq |a|$. 
\end{remark}

Novelties of this work are highlighted below. We propose hyperbolic relaxation systems for general viscous conservation laws that are entropy dissipative via a flux relaxation approach. This method leads to a simplified discretization: (i) FV methods developed for hyperbolic conservation laws can be directly used, leading to the coupling of convection and viscosity effects in the flux evolution; (ii) Discretizations to diffusion terms, which generally require an extended stencil, can be avoided. Due to the hyperbolic nature of the relaxation system, the time step $\Delta t$ of explicit schemes is determined by the Courant-Friedrichs-Lewy (CFL) condition, which reads $\Delta t=O(\Delta x)$ with $\Delta x$ denoting the cell size. It is different from the CFL condition for traditional explicit schemes (see \cite{CS1998,DY2022} for examples) of the original diffusive system, which requires $\Delta t=O(\Delta x^2)$.

The rest of this paper is organized as follows. In \Cref{secresys}, we introduce the relaxation system for general entropy dissipative system of viscous conservation laws. In \Cref{secscalar}, we develop relaxation schemes for 1-D scalar viscous conservation laws in the FV framework, and discuss the asymptotic preserving property and dissipation property. In \Cref{secNS}, the relaxation schemes are extended to the compressible Navier-Stokes equations, including both 1-D and 2-D cases. Numerical results are presented in \Cref{sectest} to illustrate the performance of the current scheme. The last section presents discussions.

\section{The relaxation system} \label{secresys}
In this section, we introduce the relaxation system for general entropy dissipative system of viscous conservation laws.
\subsection{Entropy dissipative system of viscous conservation laws}
We consider the Cauchy problem for a $d$-dimensional ($d\geq1$) entropy dissipative system in the form of viscous conservation laws \cite{Serre2010}: 
\begin{equation}
	\begin{cases}
		\pt_t \bfu+\nabla\cdot \bff(\bfu)= \nabla\cdot(\mu\bB(\bfu)\nabla \bfu), 
		\quad\quad	(x,t)\in \R^d\times\R_+,\\
		\bfu(x,0)=\bfu_0(x), \quad\quad x\in \R^d,
	\end{cases}  \label{sys}
\end{equation}
where $\bfu: \R^d\times\R_+\rightarrow \mathcal{D}\subset\R^n$ is the density function, with $\mathcal{D}$ denoting the state domain, and $\bff=(\bff_1,\dots,\bff_d):\mathcal{D}\rightarrow(\R^n)^d$ is the flux function. The viscosity tensor is denoted by $\mu\bB(\bfu)\nabla \bfu$, where $\bB=(B_{ij})_{d\times d}\in (\R^{n\times n})^{d\times d}$ is a tensor playing the role of a linear map from the space of $n\times d$ matrices into itself. The viscosity tensor is therefore
\begin{equation}
	\nabla\cdot(\mu\bB(\bfu)\nabla \bfu)
	:=\sum\limits_{1\leq i,j \leq d}\pt_{x_i}(\mu B_{ij}(\bfu)\pt_{x_j}\bfu),
\end{equation}
with $\mu$ being a positive constant and $B_{ij}:\mathcal{D}\rightarrow \R^{n\times n}$ being matrix valued functions. 

\begin{definition}
	We say that system \eqref{sys} is entropy dissipative in $\mathcal{D}$ if it is endowed with a convex entropy function $\eta(\bfu)$ and the associated entropy flux $\bg(\bfu)=(\bg_1(\bfu),\dots,\bg_d(\bfu))$ such that
	\begin{equation}
		\eta_{\bfu\bfu}(\bfu)>0,
		\quad \nabla_\bfu \bg(\bfu)=\eta_\bfu^\top(\bfu)\nabla_\bfu \bff(\bfu),
		\quad(\bB^\eta(\bfu))^\top=\bB^\eta(\bfu)\geq 0,
		\quad \forall \bfu\in\mathcal{D}. \label{conentropy}
	\end{equation} 
	Here, $\eta_\bfu(\bfu):=\nabla_\bfu \eta(\bfu)$ stands for the gradient of $\eta(\bfu)$, $\eta_{\bfu\bfu}(\bfu):=\nabla^2_{\bfu}\eta(\bfu)$ is the Hessian matrix of $\eta(\bfu)$, $\nabla_\bfu \bff=(\nabla_\bfu \bff_1,\dots,\nabla_\bfu \bff_d) \in (\R^{n\times n})^d$ is the Jacobian matrix of $\bff(\bfu)$, and 
	\begin{equation}
		\bB^\eta=(B^\eta_{ij})_{d \times d}\in (\R^{n\times n})^{d\times d}\cong \R^{nd\times nd},
		\quad B^\eta_{ij}:=\eta_{\bfu\bfu} B_{ij}\in \R^{n\times n}.
		\label{Beta}
	\end{equation}	
	We note that $\eta_{\bfu\bfu}>0$ and $(\bB^\eta)^\top=\bB^\eta \geq 0$ mean that $\eta_{\bfu\bfu}$ is positive definite and $\bB^\eta$ is symmetric positive semi-definite, respectively.
\end{definition}

Left-multiplying $\eta_\bfu^\top$ on both sides of $\eqref{sys}_1$, we observe that the smooth solution of system \eqref{sys} satisfies the entropy inequality
\begin{equation}
	\pt_t \eta+\nabla\cdot\bg-\mu\nabla\cdot\langle\eta_\bfu, \bB\nabla \bfu\rangle
	=-\mu\langle \pt_{x_i}\bfu,\eta_{\bfu\bfu} B_{ij}\pt_{x_j} \bfu\rangle\leq 0,
	\quad \forall \bfu\in\mathcal{D},  \label{entropyine}
\end{equation}
where $\langle\cdot,\cdot\rangle$ stands for the usual Euclidean inner product in $\R^n$.

\begin{remark}
	When \eqref{sys} is linear symmetric, it is also referred to as the equations of symmetric hyperbolic-parabolic type \cite{UKS1984}.
\end{remark}

\subsection{The relaxation system for entropy dissipative viscous conservation laws}
We introduce the relaxation system for \eqref{sys}:
\begin{equation}
	\left.\begin{cases}
		\pt_t \bfu+\nabla\cdot(\bC\bfv)=0, \\
		\pt_t \bfv+(\frac{\mu}{\epsilon}\bB(\bfu)+\bD)\nabla \bfu=\frac{1}{\epsilon}(\bff(\bfu)-\bC\bfv).\label{sys1}
	\end{cases}
	\right.
\end{equation}
with the initial data
\begin{equation}
	\bfu(x,0)=\bfu_0(x),\,\,\,\,\bfv(x,0)=\bfv_0(x):= \bC^{-1}[\bff(\bfu_0(x))-\mu\bB(\bfu_0)\nabla \bfu_{0}(x)].
\end{equation}
Here, $\bfv=(\bfv_1,\dots,\bfv_d)^\top\in(\R^{n})^d$ is an artificial matrix valued variable, $\epsilon$ is a small positive parameter standing for the relaxation time, $\bC=(C_{ij})_{d\times d}\in (\R^{n\times n})^{d\times d}$ and $\bD=(D_{ij})_{d\times d}\in (\R^{n\times n})^{d\times d}$ are two non-singular constant tensors to be chosen later, and $C_{ij}, D_{ij}\in \R^{n\times n}$ for $1\leq i,j \leq d$.

\begin{proposition}  \label{Leading}
	Let $\bfU=(\bfu,\bfv)^\top$ be a smooth solution of \eqref{sys1}. For small $\epsilon$, we consider the Chapman-Enskog expansion to $\bfU$:
	\begin{equation}
		\bfU(x,t)=\bfU^{(0)}(x,t)+\epsilon\bfU^{(1)}(x,t)+O(\epsilon^2). \label{expansion}
	\end{equation}
	Then, the leading order solution $\bfU^{(0)}$ satisfies 
	\begin{equation}
		\begin{aligned}
			&\bC\bfv^{(0)}=\bff(\bfu^{(0)})-\mu\bB(\bfu^{(0)})\nabla \bfu^{(0)}, \\
			&\pt_t \bfu^{(0)}+\nabla\cdot \bff(\bfu^{(0)})= \nabla\cdot (\mu\bB(\bfu^{(0)})\nabla \bfu^{(0)}). \label{equilibrium}
		\end{aligned}
	\end{equation}
\end{proposition}
\begin{proof}
	Substituting \eqref{expansion} into \eqref{sys1}, we have
	\begin{equation}
		\pt_t \bfu^{(0)}+\nabla\cdot (\bC\bfv^{(0)})
		=-\ep\big(\pt_t \bfu^{(1)}+\nabla\cdot (\bC\bfv^{(1)})\big)+O(\ep^2)=O(\ep),\label{expansion1}
	\end{equation}
	and
	\begin{equation}
		\begin{aligned}
			&\bC\bfv^{(0)}-\bff(\bfu^{(0)})+\mu\bB(\bfu^{(0)})\nabla \bfu^{(0)}\\
			=&-\ep\bC\bfv^{(1)}-\ep(\pt_t \bfv+\bD\nabla \bfu) +(\bff(\bfu)-\bff(\bfu^{(0)}))
			-(\mu\bB(\bfu)\nabla \bfu-\mu\bB(\bfu^{(0)})\nabla \bfu^{(0)})\\
			=&O(\ep)+\nabla_{\bfu}\bff(\bfu^{(0)})(\bfu-\bfu^{(0)})-(\mu\bB(\bfu)-\mu\bB(\bfu^{(0)}))\nabla \bfu-\mu\bB(\bfu^{(0)})(\nabla \bfu-\nabla \bfu^{(0)})\\
			=&O(\ep). \label{expansion2}
		\end{aligned}
	\end{equation}
	Note that \eqref{expansion2} implies $\eqref{equilibrium}_1$. Substituting \eqref{expansion2} into \eqref{expansion1} derives $\eqref{equilibrium}_2$. 
	
	The proof is completed.
\end{proof}

For $\bfU=(\bfu,\bfv)^\top\in \R^{n+nd}$, the system \eqref{sys1} can be rewritten as 
\begin{equation}
	\pt_t\bfU + \sum_{j} \bM_j(\bfU)\pt_{x_j}\bfU=\bH(\bfU).\label{sys2}
\end{equation}
The source term is $\bH(\bfU)=(\mathbf{0},\frac{1}{\ep}(\bff(\bfu)-\bC\bfv))\in \R^{n+nd}$, and the tensor $\bM_j(\bfU)$ reads
\begin{equation}
	\begin{aligned}
		\bM_j(\bfU):=
		\left(
		\begin{array}{cc}
			0			&(\wbC_j)^\top  \\
			\frac{\mu}{\epsilon}\wbB_j(\bfu)+\wbD_j\,\,			& 0  \\
		\end{array}
		\right)\in \R^{n(1+d)\times n(1+d)},\,\,\,\,\,\,\,\,\,\,\,\,1\leq j\leq d,
	\end{aligned}
\end{equation}
where
\begin{equation}
	\small
	\begin{aligned}
		\wbB_j:=(B_{1j},\dots,B_{dj}),\, \wbC_j:=(C_{1j},\dots,C_{dj}),\, \wbD_j:=(D_{1j},\dots,D_{dj})\in (\R^{n \times n})^d. \label{BCD}
	\end{aligned}
\end{equation}

For $\bfu\in\mathcal{D}$, $\bC$ and $\bD$ should be chosen suitably to ensure the hyperbolicity and the entropy dissipative nature of \eqref{sys2}. 
\begin{proposition}
	There exist non-singular constant tensors $\bC$ and $\bD$ such that the relaxation system \eqref{sys2} is hyperbolic and the following condition holds:
	\begin{equation}
		\mu\bB^\eta+\epsilon(\bD-\bC^{-1}\bA\otimes\bA)^\eta>0,
		\quad\quad\quad\forall\bfu\in\mathcal{D}.
		\label{consys}
	\end{equation}
	Here, $(\bD-\bC^{-1}\bA\otimes\bA)^\eta$ is defined in the same manner as $\bB^\eta$ $($see \eqref{Beta}$)$, where $\bA:= \nabla_\bfu \bff\in (\R^{n\times n})^d$, i.e. $\bA=(A_1,\cdots,A_d)$ with $A_i=\nabla_\bfu \bff_i\in \R^{n\times n}$, and $\bA\otimes\bA=((\bA\otimes\bA)_{ij})_{d\times d}\in (\R^{n\times n})^{d\times d}$ with $(\bA\otimes\bA)_{ij}= A_{i}A_{j}\in \R^{n\times n}$. 
	
	In addition, the system \eqref{sys1} is entropy dissipative in the sense that the smooth solution of \eqref{sys1} satisfies the entropy inequality
	\begin{equation}
		\begin{aligned}
			&\pt_t \eta+\nabla\cdot \bg- \nabla\cdot\left\langle\eta_\bfu, \big[\mu\bB+\epsilon\big(\bD-\bC^{-1}\bA\otimes\bA\big)\big]\nabla \bfu\right\rangle +O(\epsilon\mu)+O(\epsilon^2)\\
			=&-\left\langle \nabla \bfu,\eta_{\bfu\bfu}\big[\big(\mu\bB+\epsilon(\bD-\bC^{-1}\bA\otimes\bA)\big)\nabla \bfu\big]\right\rangle\\
			=&-\left\langle \nabla \bfu,\big(\mu\bB^\eta+\epsilon(\bD-\bC^{-1}\bA\otimes\bA)^\eta\big)\nabla \bfu\right\rangle\leq 0, \quad\quad\quad\forall\bfu\in\mathcal{D}. \label{entropyineq}
		\end{aligned}
	\end{equation}
\end{proposition}
\begin{remark}
	For system \eqref{sys} in the low viscosity regime, i.e. $0<\mu=\frac{1}{Re}\ll1$ with $Re$ being the Reynolds number, if we choose $\epsilon=O(\mu^{1+\alpha})$ with some constant $\alpha>0$, both terms $O(\mu\epsilon)$ and $O(\epsilon^2)$ in inequality \eqref{entropyineq} are high-order small quantities that can be neglected.
\end{remark}
\begin{proof}
	Denote by $\bI_{nd}$ the $nd \times nd$ identity matrix and let 
	\begin{equation}
		\bC=\bI_{nd},\,\,\,\bD=a^2 \bI_{nd} \label{CD},
	\end{equation}	
	where $a$ is a positive constant. Then, for $\bfu\in\mathcal{D}$, the condition \eqref{consys} is satisfied by setting $a$ large enough.
	
	Applying the Chapman-Enskog expansion to $\eqref{sys1}$ leads to
	\begin{equation}
		\begin{aligned}
			\bC\bfv=&\bff-\mu\bB\nabla \bfu-\epsilon(\pt_t\bfv+\bD \nabla \bfu)\\
			=&\bff-\mu\bB \nabla \bfu
			-\epsilon\big(\bC^{-1}\pt_t(\bff-\mu\bB\nabla \bfu)+\bD \nabla \bfu\big)+O(\epsilon^2).  \label{CEv}
		\end{aligned}
	\end{equation}
	Substituting \eqref{CEv} into $\eqref{sys1}_1$, we have
	\begin{equation}
		\begin{aligned}
			0=&\pt_t \bfu+\nabla\cdot[\bff-(\mu\bB +\epsilon \bD)\nabla \bfu -\epsilon\bC^{-1}\pt_t(\bff-\mu\bB\nabla \bfu)]+O(\epsilon^2)\\
			=&\pt_t \bfu+\nabla\cdot [\bff-(\mu\bB +\epsilon \bD)\nabla \bfu
			+\epsilon \bC^{-1}\bA\nabla\cdot(\bC\bfv)]\\
			&+\epsilon\pt_t\nabla\cdot(\bC^{-1}\mu\bB\nabla \bfu)+O(\epsilon^2)\\
			=&\pt_t \bfu+\nabla\cdot \bff
			-\nabla\cdot \big [ \big(\mu\bB+\epsilon \bD-\epsilon \bC^{-1}\big(\bA\otimes\bA\big)\big)\nabla\bfu\big]\\
			&-\epsilon \nabla\cdot [ \bC^{-1}\bA\nabla\cdot(\mu\bB\nabla \bfu)]
			+\pt_t\nabla\cdot(\bC^{-1}\mu\bB\nabla \bfu)+O(\epsilon^2)\\
			=&\pt_t \bfu+\nabla\cdot \bff
			-\nabla\cdot \left[\left(\mu\bB+\epsilon \bD-\epsilon \bC^{-1}\big(\bA\otimes\bA\big)\right)\nabla \bfu \right]+O(\epsilon\mu)+ O(\epsilon^2).\label{sys3}
		\end{aligned}
	\end{equation}
	Thus, left-multiplying $\eta_\bfu^\top$ on both sides of \eqref{sys3}, together with using \eqref{consys}, we can derive \eqref{entropyineq} by a direct calculation.
	
	Next, to show the hyperbolicity of \eqref{sys1}, it suffices to prove that for any $\xi=(\xi_1,\dots,\xi_d)\in \R^d$ satisfying $|\xi|^2=1$, the matrix $\sum^d_{j}\xi_j\bM_j$ has real eigenvalues. In the remaining part of the proof, for simplicity we employ the Einstein summation convention, summation over repeated indices. Denoting by $\lambda$ the eigenvalue of $\xi_j\bM_j$, we have
	\begin{equation}
		\begin{aligned}
			\small
			0=&\mathrm{det}(\lambda \bI_{n(d+1)}-\xi_j\bM_j)
			=\left|
			\begin{array}{cc}
				\lambda \bI_{n}			&-\xi_j(\wbC_j)^\top \\
				-\xi_j(\frac{\mu}{\epsilon}\wbB_j+ \wbD_j)\,\,	    	& \lambda \bI_{nd}  \\
			\end{array}
			\right|\\
			=&\lambda^{n(d+1)}\left|
			\begin{array}{cc}
				\bI_n-\frac{1}{\lambda^2}(\frac{\mu\xi_i\xi_j}{\epsilon}B_{ij}+a^2\bI_n)\,&0  \\
				-\frac{1}{\lambda}\xi_j(\frac{\mu}{\epsilon}\wbB_j+ a^2\bI_n)\,\,&  \bI_{nd}\\
			\end{array}
			\right|\\
			=&\lambda^{n(d-1)}\left|\lambda^2 \bI_n-\big(\frac{\mu\xi_i\xi_j}{\epsilon}B_{ij}+a^2\bI_n\big)\right|,\label{chara}
		\end{aligned}
	\end{equation}
	where we have recalled \eqref{BCD} and \eqref{CD}. Thus, it suffices to show that all the eigenvalues of the matrix $\frac{\mu\xi_i\xi_j}{\epsilon}B_{ij}+a^2\bI_n$ are positive. For $\by\in \R^n$, let $\bz=(\xi_1\by,\cdots,\xi_d\by)\in\R^{nd}$. Then, it follows from \eqref{conentropy} and \eqref{Beta} that
	\begin{equation}
		\by^\top(\xi_i\xi_j\eta_{\bfu\bfu}B_{ij})\by
		=\bz^\top\bB^\eta\bz\geq0.
	\end{equation}
	which implies that
	\begin{equation}
		\xi_i\xi_j\eta_{\bfu\bfu}B_{ij}\geq 0.  \label{Bij}
	\end{equation}
	Combining \eqref{conentropy} and \eqref{Bij} leads to
	\begin{equation}
		\frac{\mu\xi_i\xi_j}{\epsilon}\eta_{\bfu\bfu}B_{ij}+a^2\eta_{\bfu\bfu}>0. \label{temp1}
	\end{equation}
	Noting that
	\begin{equation}
		\frac{\mu\xi_i\xi_j}{\epsilon}B_{ij}+a^2\bI_n=
		\eta_{\bfu\bfu}^{-1}\big(\frac{\mu\xi_i\xi_j}{\epsilon}\eta_{\bfu\bfu}B_{ij}+a^2\eta_{\bfu\bfu}\big),\,\,\,\,\,(\eta_{\bfu\bfu}^{-1})^\top =\eta_{\bfu\bfu}^{-1}>0,
	\end{equation}
	and the fact that the matrix product of a symmetric positive matrix and a positive matrix has positive eigenvalues, we obtain from \eqref{temp1} that all the eigenvalues of $\frac{\mu\xi_i\xi_j}{\epsilon}B_{ij}+a^2\bI_n$ are positive. Consequently, all the roots of equation \eqref{chara} are real.
	
	The proof is completed.
\end{proof}

\begin{remark}
	For the case $\mu=0$ and $\bC=\bI_{nd}$, the relaxation system \eqref{sys1} recovers the model of Jin and Xin \cite{JX1995}.
\end{remark}

\section{The relaxation schemes for 1-D scalar viscous conservation laws} \label{secscalar}
In this section, we discuss the discretizations for relaxation systems of 1-D scalar viscous conservation laws. Based on the FV framework, a class of relaxation schemes is designed where the numerical fluxes can be obtained by the FV solvers developed for hyperbolic conservation laws. Later in this section, taking the relaxation schemes for the 1-D viscous Burgers equation as examples, we show the asymptotic preserving (AP) property in the coarse regime without resolving the relaxation scale, and derive the dissipation property of the modified equations. Finally, an IMEX-GRP method to deal with the stiffness of the source term is developed to close this section.

\subsection{Relaxation schemes in the finite-volume framework}
Consider the 1-D scalar viscous conservation laws
\begin{equation}
	u_t+f(u)_x=(\mu\phi(u) u_x)_{x}, \,\,\, (x,t)\in \Omega\times\R_+,\, \label{1Dsc}
\end{equation} 
where $\Omega=(0,L)$ for $L>0$, $u\in \mathcal{D}\subset\R$, $f$ is a smooth function of $u$, $\phi(u)=b^{\prime}(u)\geq0$ for a smooth function $b(u)$, and $\mu$ is a positive constant. 

With $U=(u,v)^\top$ for $v\in\R$, the relaxation system for \eqref{1Dsc} can be written as 
\begin{equation}
	\frac{\partial U}{\partial t}+\frac{\pt F(U)}{\pt x}=H(U),\,
	F(U)=\begin{pmatrix}
		v\\
		\frac{\mu}{\ep}b(u)+a^2u
	\end{pmatrix},\,
	H(U)=\begin{pmatrix}
		0\\
		\frac{1}{\ep}(f(u)-v)
	\end{pmatrix},\label{1Dsc2}
\end{equation}
where $\epsilon$ is a small positive parameter and $a$ is a constant satisfying \eqref{1Dsc3}.

Note that \eqref{1Dsc2} is a hyperbolic system with stiff relaxation terms. FV methods developed for hyperbolic conservation laws can be applied directly to \eqref{1Dsc2}. In addition, implicit treatments are required to address the stiffness of relaxation terms. In the following, we introduce a class of relaxation schemes in the FV framework.

We first equally partition the domain $\Omega$ by grid points $0=x_{-1/2}<x_{1/2}<\cdots<x_{N+1/2}=L$; we define the mesh $\{\Omega_j=(x_{j-1/2},x_{j+1/2}),\,j=0,1,\cdots,N\}$ and set the mesh size $\Delta x= x_{j+1/2}-x_{j-1/2}$. Furthermore, we denote by $\{t_n\}^\infty_{n=0}$ the sequence of discretized time levels, and let $\Delta t=t_{n+1}-t_n$. Assume that the data at time $t=t_n$ are piecewise $k$-th degree polynomials $(k\geq1)$ obtained by the data reconstruction. Then, a Godunov-type method for \eqref{1Dsc2} takes the form
\begin{equation} 
	U^{n+1}_j=U^{n}_j-\frac{\Delta t}{\Delta x}(F^{n+1/2}_{j+1/2}-F^{n+1/2}_{j-1/2})+\Delta t H^{n+1/2}_j.    \label{1DGodunov}
\end{equation} 
The term $U^n_j$ is the cell average of $U$ over $\Omega_j$ at time $t=t_n$, and  $F^{n+1/2}_{j+1/2}$ is the numerical flux across the cell interface $x=x_{j+1/2}$, which can be evaluated by a GRP solver \cite{BLW2006,BL2007}, or an approximate Riemann solver, such as a Roe-type solver \cite{Roe1981}. Moreover, $H^{n+1/2}_j$ is an approximation of the stiff source term, which can be computed based on the backward Euler method or the Crank-Nicolson time discretization. 

In following of this section, we will present several simple examples of relaxation schemes for 1-D viscous conservation laws by focusing on the 1-D viscous Burgers equation \eqref{1Dsc} with $f(u)=\frac{1}{2}u^2$ and $\phi(u)\equiv1$.
\begin{example}[A first-order upwind type relaxation scheme]\label{example1}
	Assume that the data at time $t = t_n$ are piecewise constant, i.e., we have
	\begin{equation}
		U(x,t_n)=U^n_j, \quad\quad\quad x\in(x_{j-1/2},x_{j+1/2}).  \label{reconstruction0}
	\end{equation}
	
	Applying the backward Euler method on the stiff source term, we obtain a simple first-order upwind type relaxation scheme by setting 
	\begin{equation}
		\begin{aligned}
			&F^{n+1/2}_{j+1/2}=F(U^{n}_{j+1/2}),\,\,\,\;\;
			H^{n+1/2}_j=H(U^{n+1}_j),\\
			&U^{n}_{j+1/2}=\frac{1}{2}(U^{n,l}_{j+1/2}+U^{n,r}_{j+1/2})   \label{upwind}
			-\frac{1}{2\sqrt{\mu/\ep+a^2}}M(U^{n,r}_{j+1/2}-U^{n,l}_{j+1/2})
		\end{aligned}
	\end{equation}
	in \eqref{1DGodunov}, where $U^{n,l}_{j+1/2}$ and $U^{n,r}_{j+1/2}$ are the limiting values of $U(x,t_n)$ on both sides of $(x_{j+1/2},t_n)$, and
	\begin{equation}
		M:=\frac{\pt F}{\pt U}=\begin{pmatrix}
			0\,\,\,&1\\
			\frac{\mu}{\ep}+a^2\,\,\,&0\\
		\end{pmatrix} \label{MBurgers}
	\end{equation} 
	is a constant matrix.
\end{example}

\begin{example}[A upwind type relaxation scheme with second-order accuracy in space] \label{example2}
	The relaxation scheme in \cref{example1} can be easily extended to the one with second-order accuracy in space as follows. Assume that the data at time $t = t_n$ are piecewise linear with a slope $(U_x)^n_j$, i.e., we have
	\begin{equation}
		U(x,t_n)=U^n_j+(U_x)^n_{j}(x-x_j),\quad\quad\quad x\in(x_{j-1/2},x_{j+1/2}).  \label{reconstruction1}
	\end{equation}
	For smooth flows, the following approximation can be employed,
	\begin{equation}
		(U_x)^n_j= (U^n_{j+1}-U^n_{j-1})/(2\Delta x). \label{reconstruction2}
	\end{equation} 
	Then, a upwind type relaxation scheme with second-order accuracy in space can be obtained by setting \eqref{upwind} in \eqref{1DGodunov}.
\end{example}

\subsection{Asymptotic preserving property}
In this subsection, we discuss the asymptotic preserving (AP) property of the scheme proposed in \cref{example2}.

To avoid unacceptable computational cost, a practical relaxation scheme should be able to give a solution without resolving the relaxation scales at order of $\epsilon$, which allows $\Delta x\gg \epsilon$ and $\Delta t\gg\epsilon$. 

Whether a numerical scheme is able to capture the relaxation limit is closely connected with its AP property. This concept can be traced back to the pioneering work of Larsen \cite{L1983}. For hyperbolic systems with stiff relaxation source terms \cite{J1999,LM2002}, the standard procedure of asymptotic analysis in the AP framework is to verify that a numerical scheme reduces to a consistent and stable scheme of the leading order limiting equations, where $\Delta x$ and $\Delta t$ are fixed and independent of the relaxation parameter $\epsilon$. 

A discrete numerical method for the relaxation system \eqref{1Dsc2}, denoted by $P^{\epsilon}_h$, gives an approximate solution $U_h=(u_h, v_h)$, depending on the discrete scale $h=(\Delta x,\Delta t)$. It is expected that the asymptotic behavior of $P^{\epsilon}_h$ depends not only on the relaxation scale $\epsilon$, but also on the numerical scale $h=(\Delta x,\Delta t)$. In most previous studies, the time asymptotics \cite{J1999,DP2017} and space asymptotics \cite{LMM1987,JD1996,LM2002} are usually considered independently. Studies on AP properties of fully discrete schemes (such as \cite{CX2015}) are relatively rare. 

Inspired by the unified preserving property assessing asymptotic orders of a kinetic scheme \cite{GLX2023}, we will discuss the AP property of a relaxation scheme in a ``coarse'' regime that $h=O(\epsilon^\alpha)$ for $\alpha\in(0,1)$, which is between the so-called ``thick'' regime $h=O(1)$ and the ``intermediate'' regime $h=O(\epsilon)$ \cite{LMM1987}. We give the following definition.

\begin{definition} \label{defAP}
	Let $P^{\epsilon}_h$ be a consistent numerical scheme of the relaxation system \eqref{1Dsc2}. Let $\Delta x=O(\epsilon^\alpha)$ and $\Delta t=O(\epsilon^\beta)$ with $\alpha_0<\alpha<1$ and $\beta_0<\beta<1$ for $\alpha_0\in [0,1)$ and $\beta_0\in [0,1)$. Then, we say that $P^{\epsilon}_h$ is AP in a coarse regime at the cell resolutiuon $\Delta x=o(\epsilon^{\alpha_0})$ and $\Delta t=o(\epsilon^{\beta_0})$, if $P^{\epsilon}_h\rightarrow P^0_0$ as $\epsilon\rightarrow0$.
	
\end{definition}

The above AP concept provides a picture how a relaxation scheme approaches to the asymptotic limit, as demonstrated in \cref{fig:AP}. In the following we analysis the AP property of the scheme proposed in \cref{example2} by using the so-called asymptotic modified equation (AME) approach \cite{JD1996}. 

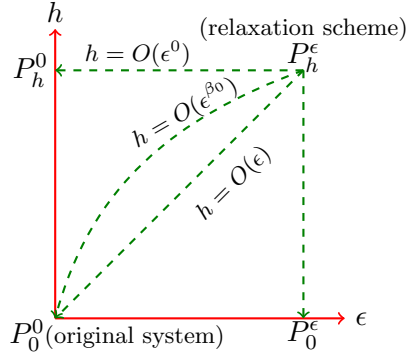
\begin{figure}[htbp]  
	\centering
	\setlength{\abovecaptionskip}{0.05cm}
	\setlength{\belowcaptionskip}{0.0cm}
	\begin{tikzpicture}[scale=1.1]
		\draw [red] [thick][->] (0,0) -- (3.5,0);
		\draw [red] [thick][->] (0,0) -- (0,3.5);
		\draw [green!50!black][thick][dashed][->] (3,3) -- (3,0);
		\draw [green!50!black][thick][dashed][->] (3,3) -- (0,3);
		\draw [green!50!black][thick][dashed][->] (3,3) -- (0,0);
		\draw [green!50!black][thick][dashed][->] (3,3) to [out=200, in=75] (0,0);
		\node at (3.7,0) {$\epsilon$};
		\node at (0,3.7) {$h$};
		\node at (0.75, -0.2) {$P^0_0$\footnotesize \text{(original system)}};
		\node at (-0.3, 3) {$P^0_h$};
		\node at (3, -0.2) {$P^\epsilon_0$};
		\node at (3.0, 3.15) {$P^\epsilon_h$};
		\node at (3.0, 3.5){ \footnotesize \text{(relaxation scheme)}};
		\node at (1.0, 3.2) {\footnotesize $h=O(\epsilon^0)$};
		\node [rotate around={30:(1.5, 1.5)}] at (0.7, 3.0) {\footnotesize $h=O(\epsilon^{\beta_0})$};
		\node [rotate around={45:(1.5, 1.5)}] at (0.8, 2.2) {\footnotesize $h=O(\epsilon)$};			
	\end{tikzpicture}   
	\caption{Schematic diagram of the asymptotic path to the limiting hydrodynamic regimes. The region below the line $h=O(\epsilon)$ suggests resolved relaxation scale. The line $h = O(\epsilon^{\beta_0})$ represents the upper limit for an AP scheme according to \cref{defAP}. The region between the two lines $h=O(\epsilon)$ and $h=O(\epsilon^{\beta_0})$ represents the coarse regime for a relaxation scheme to be AP.} \label{fig:AP}
\end{figure}

In order to obtain the modified equation of the scheme proposed in \cref{example2}, we perform the Taylor expansions on all terms in \eqref{1DGodunov} at $(x,t)=(x_j, t_n)$, leading to the following equation after some standard algebraic manipulations,
\begin{equation}
	\begin{aligned}
		\small
		\pt_t u_h+\frac{\Delta t}{2}\pt^2_t u_h
		=&-(\pt_x v_h-\frac{\Delta x^2}{12}\pt^3_x v_h)
		+\sqrt{\frac{\mu}{\epsilon}+a^2}\frac{\Delta x^3}{24}\pt^4_x u_h \\
		&+O(\Delta x^4,\frac{\Delta x^5}{\epsilon^{1/2}},\Delta t^2)\fL(U_h),\\
		\pt_t v_h+\frac{\Delta t}{2}\pt^2_t v_h
		=&-(\frac{\mu}{\epsilon}+a^2)(\pt_x u_h-\frac{1}{12}\Delta x^2\pt^3_x u_h)
		+\sqrt{\frac{\mu}{\epsilon}+a^2}\frac{\Delta x^3}{24}\pt^4_x v_h	\\
		+O(\frac{\Delta x^4}{\epsilon},&\frac{\Delta x^5}{\epsilon^{1/2}},\Delta t^2)\fL(U_h)+(1+\Delta t\pt_t+O(\Delta t^2)\fL)\left(\frac{f(u_h)-v_h}{\epsilon}\right),
	\end{aligned}   \label{fullmodified}
\end{equation}
where $\fL(\cdot)$ is a general linear operator acting on the dummy variable, representing the collection of temporal and spatial derivatives of different orders. It can take different forms at different places. The high order time derivatives of $u$ and $v$ can be replaced in terms of spatial derivatives successively by using $\eqref{fullmodified}$, as done in the modified equation approach \cite{WH1974}.

It is clear that for fixed $\ep$, as $\Delta t \rightarrow 0$ and $\Delta x \rightarrow 0$, the modified equation \eqref{fullmodified} reduces to the relaxation system \eqref{1Dsc2}, suggesting that the scheme in \cref{example2} is a consistent scheme (second order in space and first order in time) for \eqref{1Dsc2}.

Then, the behavior of this scheme for small $\ep$ while holding $h$ fixed is as follows. Multiplying $\epsilon$ on both sides of $\eqref{fullmodified}_2$, we can rewrite $\eqref{fullmodified}_2$ as 
\begin{equation}
	(1+O(\epsilon,\Delta x,\Delta t)\fL)v_h=f(u_h)-\mu\pt_x u_h
	+O(\epsilon,\Delta x^2,\Delta t)\fL(u_h), 
\end{equation}
which suggests that 
\begin{equation}
	\begin{aligned}
		v_h=f(u_h)-\mu\pt_x u_h+O(\epsilon,\Delta x^2,\Delta t)\fL(u_h).     \label{full3}
	\end{aligned}
\end{equation}
Thus, it follows from \eqref{full3} that
\begin{equation}
	\begin{aligned}
		\pt_x v_h=&\pt_x(f(u_h)-\mu\pt_x u_h)+\pt_x(v_h-f(u_h)+\mu\pt_x u_h)\\
		=&\pt_x(f(u_h)-\mu\pt_x u_h)+O(\epsilon,\Delta x^2,\Delta t). \label{full4}
	\end{aligned}
\end{equation}
Here, and in the sequel, we neglect $\fL(u_h)$ for simplicity. Applying $\pt_t$ and $\pt_x\pt_t$ on both sides of $\eqref{fullmodified}_1$ and $\eqref{fullmodified}_2$, respectively, we get
\begin{equation}
	\begin{aligned}
		\pt_t^2 u_h=&-\pt_x\pt_t v_h+O(\Delta x^2,\Delta x^3\epsilon^{-1/2},\Delta t)\\
		=&-\pt_x\pt_t(v_h-f(u_h)+\mu \pt_x u_h)
		+\pt_x\pt_t(\mu \pt_x u_h-f(u_h))\\
		&+O(\ep,\Delta x^2,\Delta x^3\epsilon^{-1/2},\Delta t)\\
		=&-\pt_x\pt_tf(u_h)+\mu\pt^2_x\pt_t u_h+O(\ep,\Delta x^2,\Delta x^3\epsilon^{-1/2},\Delta t),\label{full5}
	\end{aligned}
\end{equation}
where we have used \eqref{full3} to derive the last line of \eqref{full5}. Substituting \eqref{full4} and \eqref{full5} into $\eqref{fullmodified}_1$, we can obtain
\begin{equation}
	\small
	\begin{aligned}
		\pt_t u_h+\pt_x f(u_h)-\mu\pt^2_x u_h
		=\sqrt{\frac{\mu}{\epsilon}+a^2}\frac{\Delta x^3}{24}\pt^4_x u_h
		+O(\epsilon,\Delta x^2,\Delta t,\frac{\Delta x^5}{\epsilon^{1/2}},\frac{\Delta x^3\Delta t}{\epsilon^{1/2}}), \label{full6}
	\end{aligned}
\end{equation}
which leads to the following result.
\begin{proposition}  \label{APfull}
	For $\epsilon\ll1$, suppose that 
	\begin{align}
		\Delta x=O(\epsilon^\alpha),\quad
		\Delta t=\frac{CFL\Delta x}{\sqrt{\mu/\epsilon+a^2}}=O(\epsilon^{\beta}), \quad
		CFL\in (0,1),    \label{fullCFL}
	\end{align}
	with $\frac{1}{6}<\alpha<\frac{1}{2}$ and $\frac{2}{3}<\beta=\alpha+\frac{1}{2}<1$. Then, the scheme proposed in \cref{example2} is AP in a coarse regime at the cell resolutiuon $\Delta x=o(\epsilon^{1/3})$ and $\Delta t=o(\epsilon^{2/3})$. 
\end{proposition}
\begin{proof}
	With the condition \eqref{fullCFL}, the modified equation \eqref{full6} is rewritten as 
	\begin{equation}
		\pt_t u_h+\pt_x f(u_h)-\mu\pt^2_x u_h
		=C\epsilon^{3\alpha-\frac{1}{2}}\pt^4_x u_h
		+O(\epsilon,\epsilon^{2\alpha},\epsilon^{\alpha+\frac{1}{2}},\epsilon^{5\alpha-\frac{1}{2}},\epsilon^{4\alpha}), 
	\end{equation}
	where $C=O(1)$ is a constant. Note that $3\alpha-\frac{1}{2}>0$, implying that $\epsilon^{3\alpha-\frac{1}{2}}=o(1)$. Then, we apply the Chapman-Enskog expansion to $u_h$,
	\begin{equation}
		u_h=u^{(0)}_h+\epsilon u^{(1)}_h+O(\epsilon^2).
	\end{equation}
	It can be checked straightforwardly that the $u^{(0)}_h$ satisfies \eqref{1Dsc}. This implies $P^\ep_h\rightarrow P^0_0$ as $\ep\rightarrow0$. The proof is completed.
\end{proof}	

\begin{remark}
	It can be observed that for the inviscid case where $\mu=0$, the relaxation scheme in \cref{example2} is an AP scheme in the ``thick'' regime $h=O(1)$ and the governing equation of leading order terms is exactly the inviscid Burgers equation. In addition, it can be seen in a similar manner that the scheme proposed in \cref{example1} is not AP, suggesting that second-order accuracy in space is necessary for such upwind type relaxation scheme to be AP in a coarse regime.
\end{remark}


\subsection{Dissipation property}
In this subsection, we derive the dissipation property of the upwind type relaxation scheme proposed in \cref{example1} by using the AME approach as in the previous subsection.

For fixed $\ep$, based on the method of Taylor expansion, the modified equation of the scheme in \cref{example1} is given by
\begin{equation}
	\begin{aligned}
		\pt_t u_h+\frac{\Delta t}{2}\pt^2_t u_h=&-\pt_x v_h +\sqrt{\frac{\mu}{\epsilon}+a^2}\frac{\Delta x}{2}\pt^2_x u_h
		+O(\Delta x^2,\Delta t^2)\fL(U_h),\\
		\pt_t v_h+\frac{\Delta t}{2}\pt^2_t v_h
		=&-(\frac{\mu}{\epsilon}+a^2)\pt_x u_h
		+\sqrt{\frac{\mu}{\epsilon}+a^2}\frac{\Delta x}{2}\pt^2_x v_h+O(\Delta x^2, \Delta t^2)\fL(U_h)
		\\
		&+(1+\Delta t\pt_t+O(\Delta t^2)\fL)\left(\frac{f(u_h)-v_h}{\epsilon}\right).
	\end{aligned}   \label{exmo}
\end{equation}
Then, similarly, by using the Chapman-Enskog expansion for small $\ep$ (while holding $h$ fixed), we obtain from \eqref{exmo} that
\begin{equation} 
	\begin{aligned}
		\pt_t u_h +\pt_x f(u_h)=&\pt_x\big(\big(\mu+\epsilon(a^2-f^\prime(u_h)^2)
		+\sqrt{\frac{\mu}{\epsilon}+a^2}\frac{\Delta x}{2}-f^{\prime}(u_h)^2\frac{\Delta t}{2}\big)\pt_x u_h \big)\\
		&-\mu(\epsilon+\frac{3}{2}\Delta t)\pt_x^2 \pt_t u_h
		+O(\epsilon^{\frac{1}{2}}\Delta x,\epsilon\Delta t,\Delta x^2,\Delta t^2,\ep^2). \label{exmo1}
	\end{aligned}
\end{equation}
This leads to the following proposition.
\begin{proposition}
	Suppose that the condition \eqref{1Dsc3} holds and suppose that
	\begin{equation} \label{CFLcon}
		\Delta t= \frac{CFL \Delta x}{\sqrt{\frac{\mu}{\epsilon}+a^2}},
		\quad CFL\in(0,1).
	\end{equation}
	Then, the relaxation scheme proposed in \cref{example1} has a dissipation property in the sense that the governing equation of up to first order terms in the modified equation \eqref{exmo1} is dissipative. 
\end{proposition}
\begin{proof}
	Recalling \eqref{exmo1}, it suffices to show that
	\begin{equation}
		\begin{aligned}
			&\mu+\epsilon(a^2-f^\prime(u_h)^2)
			+\sqrt{\frac{\mu}{\epsilon}+a^2}\frac{\Delta x}{2}-f^{\prime}(u_h)^2\frac{\Delta t}{2}\\
			=&\mu+\epsilon(a^2-f^\prime(u_h)^2)
			+\left(\frac{\mu}{\epsilon}+a^2\right)^{-1/2}
			\left(\frac{\mu}{\epsilon}+a^2-CFL \cdot f^{\prime}(u_h)^2\right)\frac{\Delta x}{2}	\geq 0,
		\end{aligned}
	\end{equation}
	which is ensured by the condition \eqref{1Dsc3} and \eqref{CFLcon}. The proof is completed.
\end{proof}

\subsection{An implicit-explicit generalized Riemann problem method}
In this subsection, based on the acoustic GRP method \cite{BF2003,BLW2006,BL2007}, we extend the scheme proposed in \cref{example2} to a LW type one with second-order accuracy in both space and time. Such relaxation scheme enjoys the same advantages as LW type methods, such as the temporal-spatial coupling nature and relatively compact stencils. To overcome the stiffness of the source term, a novel IMEX-GRP method is developed as follows.

With piecewise linear data at $t = t_n$ as in \eqref{reconstruction1}, recalling \eqref{1DGodunov} and setting
\begin{equation} \label{exampleB0}
	\begin{aligned}
		&F^{n+1/2}_{j+1/2}=F(U^{n+1/2}_{j+1/2}),\quad
		H^{n+1/2}_j=\frac{1}{2}\big(H(U^n_j)+H(U^{n+1}_j)\big),\\
		&U^{n+1/2}_{j+1/2}=U^{n}_{j+1/2}+\frac{\Delta t}{2}(U_t)^{n+1/4}_{j+1/2},
	\end{aligned}
\end{equation}
we get a second-order LW type relaxation scheme for \eqref{1Dsc}. The Riemann solution $U^{n}_{j+1/2}$ is determined by the upwind scheme,
\begin{equation}
	U^{n}_{j+1/2}=\frac{1}{2}(U^{n,l}_{j+1/2}+U^{n,r}_{j+1/2})
	-\frac{1}{2\sqrt{\mu/\ep+a^2}}M(U^{n,r}_{j+1/2}-U^{n,l}_{j+1/2}),
\end{equation}	
where $U^{n,l}_{j+1/2}$ and $U^{n,r}_{j+1/2}$ stand for the limiting values of $U(x,t_n)$ on both sides of $(x_{j+1/2},t_n)$, and $M=\frac{\pt F}{\pt U}$ is a constant matrix given in \eqref{MBurgers}.

The term $(U_t)^{n+1/4}_{j+1/2}$ stands for the value of $U_t(x,t)$ at the cell interface $x=x_{j+1/2}$ averaged over the time interval $[t_n,t_{n+1/2}]$, and it is evaluated by solving the GRP at each singularity point $(x_{j+1/2},t_n)$,
\begin{equation}
	\left.\begin{cases}
		\begin{aligned}
			&U_t+M U_x=H(U), \\
			&U(x,t_n)=\left.\begin{cases}
				U^{n,l}(x),\,\,\,\,\,\quad x<x_{j+1/2},\\
				U^{n,r}(x),\,\,\,\,\,\quad x>x_{j+1/2},
			\end{cases}	
			\right.
		\end{aligned}
	\end{cases}\right.   \label{GRPB}
\end{equation}    
where $U^{n,l}(x)$ and $U^{n,r}(x)$ are two linear polynomials defined on the neighboring computational cells, respectively. In the spirit of the acoustic GRP methods \cite{BF2003,BLW2006,BL2007}, which is also called arbitrary derivative (ADER) method in \cite{TT2006}, $(U_t)^{n+1/4}_{j+1/2}$ is computed with the formulae
\begin{equation}
	\begin{aligned}
		(U_t)^{n+1/4}_{j+1/2}=-(F(U)_x)^n_{j+1/2}+H(U^{n+1/2}_{j+1/2}),
	\end{aligned}\label{convective}
\end{equation}
where
\begin{equation}
	\begin{aligned}
		&-(F(U)_x)^n_{j+1/2}
		=-R\Lambda^+R^{-1}(U_x)^n_j-R\Lambda^-R^{-1}(U_x)^n_{j+1}. 
	\end{aligned}
\end{equation}
The matrices $R$ and $\Lambda$ comes from decomposition of $M=\frac{\pt F}{\pt U}$, where $R$ is the right eigenmatrix of $M$, $\Lambda =\diag(\lambda_1,\lambda_2)$, $\lambda_1\, (i= 1, 2)$, are the eigenvalues of $M$, $\Lambda^+=\diag(\max(\lambda_i,0))$ and $\Lambda^-=\diag(\min(\lambda_i,0))$. 

Finally, we update the slope as follows,
\begin{equation}\label{exampleBend}
	\begin{aligned}
		&(U_x)^{n+1}_j=\frac{1}{\Delta x}\left(U^{n+1,-}_{j+1/2}-U^{n+1,-}_{j-1/2}\right),\,\,\,\,\,
		U^{n+1,-}_{j+1/2}=U^{n}_{j+1/2}+\Delta t(U_t)^{n+1/2}_{j+1/2},\\
		&(U_t)^{n+1/2}_{j+1/2}
		=-R\Lambda^+R^{-1}(U_x)^n_j-R\Lambda^-R^{-1}(U_x)^n_{j+1}+H(U^{n+1,-}_{j+1/2}).
	\end{aligned}
\end{equation}

\begin{remark}
	The implicity of this IMEX-GRP method lies in the process of computing the mid-point value $U^{n+1/2}_{j+1/2}$. Recalling \eqref{1Dsc2}, we emphasize that there is no source term in the governing equation of $u$, so $u^{n+1/2}_{j+1/2}$ can be obtained explicitly. In addition, the stiff source term $\frac{f(u)-v}{\epsilon}$ is linear in $v$, such that the computing of $v^{n+1/2}_{j+1/2}$ is implicit but linear, leading to a simple and explicit implementation.
\end{remark}

\section{The relaxation schemes for the Navier-Stokes equations} \label{secNS}
In this section, with the compressible Navier-Stokes equations as a prominent example, we extend the flux relaxation approach to multi-dimensional entropy dissipative systems of viscous conservation laws. As an example, a second-order LW type relaxation scheme in the FV framework will be proposed, where the numerical fluxes will be constructed by the IMEX-GRP method developed in \eqref{exampleB0}--\eqref{exampleBend}.
\subsection{The relaxation schemes for the 1-D Navier-Stokes equations}
\subsubsection{The 1-D Navier-Stokes equations}
We consider the 1-D compressible Navier-Stokes equations for viscous and heat conducting ideal polytropic gases, 
\begin{equation}
	\left.\begin{cases}
		\rho_t+(\rho u)_x=0,\\
		(\rho u)_t+(\rho u^2+p)_x=((\frac{4}{3}\mu+\mu^{\prime}) u_{x})_x,\\
		(\rho E)_t+(u(\rho E+p))_x=(\mu u u_x+\kappa T_{x})_x.
	\end{cases}\label{CNS}
	\right.
\end{equation}
Here, $\rho$, $u$, $p$, $E$ and $T$ are the density, velocity, pressure, specific total energy, and temperature, respectively. The pressure $p$ is computed from the equation of state, i.e. $p=(\gamma-1)\rho(E-\frac{1}{2}u^2)$, where $\gamma>1$ is the ratio of specific heats. The temperature $T$ is given by $T=\frac{p}{\rho}$. $\mu$ is the dynamic viscosity, $\mu^\prime$ is the bulk viscosity and $\kappa$ is the heat conductivity. Using the Stokes hypothesis, $\mu^\prime=0$. The heat conductivity coefficient $\kappa$ is calculated through the viscosity coefficient $\mu$ by utilizing the Prandtl number $Pr$, i.e. $\kappa =\frac{\gamma\mu}{(\gamma-1)Pr}$. For the viscosity coefficient, $\mu$ can take any reasonable form. In the present paper, we consider only the case that $\mu$ is a constant.

Denote the conservative and primitive variables by $w=(\rho,m,\fE)^\top$ and $\tw=(\rho, u, T)^\top$, respectively, where $m =\rho u$ and $\fE = \rho E$. Then, \eqref{CNS} can be rewritten as 
\begin{equation}
	\begin{aligned}
		&w_t + f(w)_x = (B(w)w_x)_x, \\
		f(w)=\begin{pmatrix}
			m\\
			um+p\\
			u(\fE+p)\\
		\end{pmatrix},\quad
		&B(w)=\widetilde{B}(\tw(w))\frac{\pt \widetilde{w}(w)}{\pt w},\quad
		\widetilde{B}(\tw)=\begin{pmatrix}
			0&0&0\\
			0&\frac{4}{3}\mu&0\\
			0&\frac{4}{3}\mu u&\kappa
		\end{pmatrix}.
	\end{aligned}\label{CNS1}
\end{equation}

\subsubsection{Entropy dissipative property}
The system \eqref{CNS1} is equipped with an entropy-entropy flux pair $(\eta(w),g(w))$. This pair takes the form
\begin{equation}
	(\eta(w), g(w))=(-\rho S,\,-\rho u S),  \label{entropy}
\end{equation}
where $\eta(w)$ is the mathematical entropy and $S=\frac{1}{\gamma-1}\ln(T\rho^{1-\gamma})$ is the physical entropy. It can be shown by a direct calculation that 
\begin{equation}
	\eta_{ww}(w)=(\frac{\pt \widetilde{w}(w)}{\pt w})^\top Q(w)\frac{\pt \widetilde{w}(w)}{\pt w},\,\,\,
	\eta_{ww}(w)B(w)=(\frac{\pt \widetilde{w}(w)}{\pt w})^\top\widehat B(w)\frac{\pt \widetilde{w}(w)}{\pt w},
\end{equation}
where 
\begin{equation}
	Q(w)=\diag\{\frac{1}{\rho},\frac{\rho}{T},\frac{\rho}{(\gamma-1)T^2}\}>0,\,\,\,\,\,
	\widehat B(w)=\diag\{ 0,\frac{4\mu}{3T},\frac{\kappa}{T^2}\}\geq 0.		
\end{equation}
This implies 
\begin{equation}
	\eta_{ww}(w)>0,\,\,\,(\eta_{ww}(w)B(w))^\top=\eta_{ww}(w)B(w)\geq 0.
\end{equation}
Therefore, the physically admissible solution of \eqref{CNS1} satisfies the entropy inequality
\begin{equation}
	\eta(w)_t+g(w)_x-(\eta_w(w)B(w)w_x)_x=-\frac{4\mu u_x^2}{3T}-\frac{\kappa T_x^2}{ T^2}\leq 0.
\end{equation}

\subsubsection{The relaxation scheme} \label{example4}
By setting $C=I_3$ and $D=a^2 I_3$ in \eqref{sys1}, we can write the relaxation system of the 1-D Navier-Stokes equation \eqref{CNS1} in the form
\begin{equation}
	\begin{aligned}		
		&U_t+M(U)U_x = H(U),\\
		&U=\begin{pmatrix}
			w\\
			v
		\end{pmatrix},\,\,\,
		M(U)=
		\begin{pmatrix}
			0			& I_3  \\
			\frac{B(w)}{\epsilon}+a^2I_3\,\,	& 0 
		\end{pmatrix},\,\,\,
		H(U)=\begin{pmatrix}
			0\\
			\frac{f(w)-v}{\epsilon}
		\end{pmatrix}. 	\label{CNS2} 
	\end{aligned}
\end{equation}
Here, $I_3$ denotes the $3\times 3$ identity matrix and $v=(v_1,v_2,v_3)^\top$. The small positive parameter $\epsilon$ stands for the relaxation time and $a$ is a positive constant satisfying
\begin{equation}
	\eta_{ww}(w)B(w)+\epsilon(a^2\eta_{ww}(w)-\eta_{ww}(w)A(w)^2)>0,
\end{equation}
where $A(w)=\frac{\pt f(w)}{\pt w}$ and $\eta(w)$ is the entropy function given in \eqref{entropy}.

The relaxation system \eqref{CNS2} can be rewritten as
\begin{equation}
	\begin{aligned}
		&U_t+F(U)_x = H(U),\,\,
	\end{aligned}\label{CNS3}
\end{equation}
where
\begin{equation}
	\begin{aligned}
		&F(U)=(v_1, v_2, v_3, a^2\rho, \frac{4\mu}{3\epsilon} u+a^2m,\frac{\kappa}{\epsilon} T+\frac{2\mu }{3\epsilon}u^2+a^2\fE)^\top.  
	\end{aligned}
\end{equation}
Relaxation schemes for \eqref{CNS3} can be developed in the FV framework as in the previous section, and a second-order LW type relaxation scheme is presented as follows.

With the same notations as in Section 3, suppose that the data at time $t=t_n$ are piecewise linear with a slope $(U_x)^n_j$, i.e., we have
\begin{equation}
	U(x,t_n)=U^n_j+(U_x)^n_{j}(x-x_j), \;\;\; x\in(x_{j-1/2},x_{j+1/2}). \label{reconstruction}
\end{equation}
Then, a second-order LW type scheme for \eqref{CNS3} can be written as
\begin{equation}
	\begin{aligned}
		&U^{n+1}_j=U^{n}_j-\frac{\Delta t}{\Delta x}
		\left(F^{n+1/2}_{j+1/2}-F^{n+1/2}_{j-1/2}\right)
		+\frac{\Delta t}{2}\left(H^n_j+H^{n+1}_j\right),  \\
		&F^{n+1/2}_{j+1/2}=F\left(U^{n+1/2}_{j+1/2}\right),\,\,\,\,\,
		H^n_j=H\left(U^n_j\right),  \\
		&U^{n+1/2}_{j+1/2}=U^{n}_{j+1/2}+\frac{\Delta t}{2}(U_t)^{n+1/4}_{j+1/2}. 
	\end{aligned}\label{1Dflux}
\end{equation}
The Riemann solution $(U)^{n}_{j+1/2}$ is determined by solving the Riemann problem (RP),
\begin{equation}
	\begin{cases}
		\begin{aligned}
			&U_t+F(U)_x = 0, \\
			&U(x,t_n)=\left.\begin{cases}
				U^{n,l}(x_{j+1/2}-0),\,\,\,\,\,\quad x<x_{j+1/2},\\
				U^{n,r}(x_{j+1/2}+0),\,\,\,\,\,\quad x>x_{j+1/2},   
			\end{cases}	
			\right.
		\end{aligned}
	\end{cases}   \label{1DRP}
\end{equation}
and the time derivative $(U_t)^{n+1/4}_{j+1/2}$ is obtained by solving the GRP at $(x_{j+1/2},t_n)$,
\begin{equation}
	\left.\begin{cases}
		\begin{aligned}
			&U_t+F(U)_x=H(U), \\
			&U(x,t_n)=\left.\begin{cases}
				U^{n,l}(x),\,\,\,\,\,\quad  x<x_{j+1/2},\\
				U^{n,r}(x),\,\,\,\,\,\quad  x>x_{j+1/2},
			\end{cases}	
			\right.
		\end{aligned}
	\end{cases}\right.   \label{1DGRP}
\end{equation}    
Here, $U^{n,l}(x)$ and $U^{n,r}(x)$ are two linear polynomials defined on the neighboring computational cells, respectively. The RP \eqref{1DRP} can be solved by approximate Riemann solvers, such as a Roe-type solver \cite{Roe1981} or HLLC-type solver \cite{BCLC1997}. 

Next, in analogy with the acoustic GRP solver \cite{BF2003,BLW2006,BL2007}, we linearize the convective terms in \eqref{1DGRP} around the approximate local Riemann solution $U^{n}_{j+1/2}$ to derive that
\begin{equation}
	\left.\begin{cases}
		\begin{aligned}
			&\pt_t U+ M\pt_x U=H(U), \\
			&w(x,t_n)=\left.\begin{cases}
				U^{n,l}(x),\,\,\,\,\,\quad x<x_{j+1/2},\\
				U^{n,r}(x),\,\,\,\,\,\quad x>x_{j+1/2},
			\end{cases}	
			\right.
		\end{aligned}
	\end{cases}\right.
\end{equation}
where $M:=\frac{\pt F}{\pt U}(U^{n}_{j+1/2})$. Then, the IMEX-GRP solver takes the form
\begin{equation}
	(U_t)^{n+1/4}_{j+1/2}
	=-R\Lambda^+R^{-1}(U_x)^n_j-R\Lambda^-R^{-1}(U_x)^n_{j+1}+H(U^{n+1/2}_{j+1/2}), \label{1DNSIMEXGRP}
\end{equation}
where $\Lambda =\diag(\lambda_1,\dots,\lambda_6)$, $\lambda_i$, $i= 1,\dots, 6$, are the eigenvalues of $M$, $R$ is the right eigenmatrix of $M$, $\Lambda^+=\diag(\max(\lambda_i,0))$ and $\Lambda^-=\diag(\min(\lambda_i,0))$.

Finally, we update the slope as follows,
\begin{equation}
	\begin{aligned}
		&(U_t)^{n+1/2}_{j+1/2}
		=-R\Lambda^+R^{-1}(U_x)^n_j-R\Lambda^-R^{-1}(U_x)^n_{j+1}+H(U^{n+1,-}_{j+1/2}),
		\\
		&U^{n+1,-}_{j+1/2}=U^{n}_{j+1/2}+\Delta t(U_t)^{n+1/2}_{j+1/2},\,\,\,\,\,\,\,
		(U_x)^{n+1}_j=\frac{1}{\Delta x}\left(U^{n+1,-}_{j+1/2}-U^{n+1,-}_{j-1/2}\right).
		\label{1Dslope}
	\end{aligned}
\end{equation}

\subsection{The relaxation schemes for the 2-D Navier-Stokes equations}
\subsubsection{The 2-D Navier-Stokes equations}
We consider the 2-D compressible Navier-Stokes equations for viscous and heat conducting ideal polytropic gases,
\begin{equation}
	\left.\begin{cases}
		\pt_t\rho+\nabla\cdot (\rho\bfu)=0,\\
		\pt_t(\rho\bfu)+ \nabla\cdot (\rho\bfu\otimes \bfu+p\bI_2) =\nabla\cdot \mathbf{\tau},\\
		\pt_t(\rho E)+\nabla\cdot\big((\rho E+p)\bfu\big)= \nabla\cdot\bfq+\nabla\cdot(\mathbf{\tau}\bfu).
	\end{cases}\label{2DCNS}
	\right.
\end{equation}
Here, in addition to the thermodynamical flow variables $\rho$, $p$ and $T=\frac{p}{\rho}$, $\bfu=(u_X,u_Y)^\top$ is the velocity, $E=(\frac{p}{\gamma-1}+\frac{1}{2}(u^2_X+u^2_Y))$ is the specific total energy. The viscous stress tensor $\mathbf{\tau}$ and the heat flux vector $\bfq$ are given by
\begin{equation}
	\mathbf{\tau}=\mu\big(\nabla \bfu+ (\nabla \bfu)^\top-\frac{2}{3}(\nabla\cdot\bfu) \bI_2\big),\,\,\,\,\,\,\bfq=\kappa\nabla T,
\end{equation}
where $\mu$ and $\kappa$ are the same as in the 1-D case.

Denote the conservation and hydrodynamic variables by $\bw=(\rho,\bfm,\fE)^\top$ and $\tbw=\tbw(\bw)=(\rho,\bfu, T)^\top$, respectively, where $\bfm=(m_X,m_Y)^\top=(\rho u_X,\rho u_Y)^\top$ and $\fE=\rho E$. Let $\bff(\bw):=(\bff_X(\bw),\bff_Y(\bw))$ and $\tbB(\tbw):=(\tB_{\alpha\beta}(\tbw))_{2\times2}$ with $\alpha, \beta=X, Y$, where $\bff_X(\bw)=(m_X,m_Xu_X+p,m_Yu_X,(\fE+p)u_X)^\top$, $\bff_Y(\bw)=(m_Y,m_Xu_Y,m_Yu_Y+p,(\fE+p)u_Y)^\top$,
and
\begin{equation}
	\tB_{XX}(\tbw)=\begin{pmatrix}
		0\,&0\,&0\,&0\,\\
		0\,&\frac{4}{3}\mu\,&0\,&0\,\\
		0\,&0\,&\mu\,&0\,\\
		0\,&\frac{4}{3}\mu u_X\,&\mu u_Y\,&\kappa\,\\
	\end{pmatrix},
	\,\,
	\tB_{XY}(\tbw)=\begin{pmatrix}
		0\,&0\,&0\,&0\,\\
		0\,&0\,&-\frac{2}{3}\mu\,&0\,\\
		0\,&\mu\,&0\,&0\,\\
		0\,&\mu u_Y\,&-\frac{2}{3}\mu u_X\,&0\,\\
	\end{pmatrix},
\end{equation}
\begin{equation}
	\tB_{YX}(\tbw)=\begin{pmatrix}
		0\,&0\,&0\,&0\,\\
		0\,&0\,&\mu\,&0\,\\
		0\,&-\frac{2}{3}\mu\,&0\,&0\,\\
		0\,&-\frac{2}{3}\mu u_Y\,&\mu u_X\,&0\,\\
	\end{pmatrix},
	\,\,
	\tB_{YY}(\tbw)=\begin{pmatrix}
		0\,&0\,&0\,&0\,\\
		0\,&\mu\,&0\,&0\,\\
		0\,&0\,&\frac{4}{3}\mu\,&0\,\\
		0\,&\mu u_X\,&\frac{4}{3}\mu u_Y\,&\kappa\,\\
	\end{pmatrix}.\\
\end{equation}
Then, \eqref{2DCNS} can be rewritten as 
\begin{equation}
	\pt_t\bw+\nabla\cdot \bff(\bw)=\nabla\cdot(\bB(\bw)\nabla \bw), \label{2DCNS2}
\end{equation}
where $\bB(\bw)=(B_{\alpha\beta}(\bw))_{2\times 2}$ with $B_{\alpha\beta}(\bw)=\tB_{\alpha\beta}(\tbw(\bw))\frac{\pt \tbw(\bw)}{\pt\bw}$. 

Here we remark that the entropy dissipation property of \eqref{2DCNS2} can be derived in a similar manner as the 1-D case, and the details are omitted.

\subsubsection{The relaxation scheme} \label{example5}
Similarly, with artificial variables $\bfv_X, \bfv_Y\in \R^4$, the relaxation system of \eqref{2DCNS2} takes the form
\begin{equation}
	\left.\begin{cases}
		\pt_t\bw+\pt_x\bfv_X + \pt_y \bfv_Y =0,\\
		\pt_t\bfv_X+\frac{1}{\epsilon}\tB_{XX} \pt_x \widetilde{\bw}+a^2\pt_x \bw
		+\frac{1}{\epsilon}\tB_{XY} \pt_y \widetilde{\bw}
		= \frac{1}{\epsilon}(\bff_X-\bfv_X),\\
		\pt_t\bfv_Y+\frac{1}{\epsilon}\tB_{YX} \pt_x \widetilde{\bw} 
		+\frac{1}{\epsilon}\tB_{YY} \pt_y \widetilde{\bw}+a^2\pt_y \bw
		= \frac{1}{\epsilon}(\bff_Y-\bfv_Y).
	\end{cases}	\label{2DCNS3}
	\right.
\end{equation}
Here, $\epsilon$ is a small positive parameter and $a$ is a positive constant satisfying
\begin{equation}
	\bB^\eta(\bw)
	+\epsilon(a^2 \bI_{8}- \bA(\bw)\otimes\bA(\bw))^\eta>0,
\end{equation}
where $\bA(\bw)=\nabla_\bw \bff(\bw)$, $\bB^\eta(\bw)$ and $(a^2 \bI_{8}- \bA(\bw)\otimes\bA(\bw))^\eta$ are defined in the same manner as in \eqref{Beta}, and $\eta(\bw)$ is the mathematical entropy given as in \eqref{entropy}. Let
\begin{equation}
	\begin{aligned}
		&\bF_{\bfv_X}(\bw)=(0,\frac{4\mu}{3\epsilon} u_X,\frac{\mu}{\epsilon} u_Y,\kappa T+\frac{\mu}{2\epsilon}(\frac{4}{3}u^2_X+u^2_Y))^\top+a^2\bw^\top,\\
		&\bF_{\bfv_Y}(\bw)=(0,-\frac{2\mu}{3\epsilon} u_Y, \frac{\mu}{\epsilon} u_X,0)^\top,\quad
		\bG_{\bfv_X}(\bw)=(0,\frac{\mu}{\epsilon} u_Y, -\frac{2\mu}{3\epsilon} u_X,0)^\top,\\
		&\bG_{\bfv_Y}(\bw)=(0,\frac{\mu}{\epsilon} u_X,\frac{4\mu}{3\epsilon} u_Y,\kappa T+\frac{\mu}{2\epsilon}(u^2_X+\frac{4}{3}u^2_Y))^\top+a^2\bw^\top.	
	\end{aligned}
\end{equation}
Then, the relaxation system \eqref{2DCNS3} can be rewritten as 
\begin{equation}
	\pt_t \bfU +\pt_x \bF(\bfU^X)+\pt_y \bG(\bfU^Y)+\bh(\tbw,\pt_x\tbw,\pt_y \tbw)=\bH(\bfU), \label{2DCNS4}
\end{equation}
where
\begin{equation}
	\begin{aligned}
		&\bfU=(\bw,\bfv_X,\bfv_Y)^\top, \,\,\,\,\,\bfU^X=(\bw,\bfv_X)^\top,\,\,\,\,\,\bfU^Y=(\bw,\bfv_Y)^\top,\\
		&\bF(\bfU^X)=(\bfv_X,\bF_{\bfv_X}(\bw),\bF_{\bfv_Y}(\bw))^\top,\,\,\,\,\,
		\bG(\bfU^Y)=(\bfv_Y,\bG_{\bfv_X}(\bw),\bG_{\bfv_Y}(\bw))^\top,\\
		&\bH(\bfU)=(\mathbf{0},\frac{1}{\epsilon}(\bff_X(\bw)-\bfv_X),
		\frac{1}{\epsilon}(\bff_Y(\bw)-\bfv_Y))^\top,\\
		&\bh(\tbw,\pt_x \tbw,\pt_y \tbw)
		=(\mathbf{0},\bh_{\bfv_X}(\tbw,\pt_y \tbw),\bh_{\bfv_Y}(\tbw,\pt_x \tbw))^\top,\\
		&\bh_{\bfv_X}(\tbw,\pt_y \tbw)
		=(0,0,0,\frac{\mu}{\epsilon}u_Y\pt_y u_X-\frac{2}{3}\frac{\mu}{\epsilon}u_X\pt_y u_Y)^\top,\\
		&\bh_{\bfv_Y}(\tbw,\pt_x \tbw),
		=(0,0,0,\frac{\mu}{\epsilon}u_X\pt_x u_Y-\frac{2}{3}\frac{\mu}{\epsilon}u_Y\pt_x u_X)^\top.
	\end{aligned}
\end{equation}
Note that $\bF(\bfU^X)$ is independent of $\bfv_Y$ and $\bG(\bfU^Y)$ is independent of $\bfv_X$.

We continue to develop relaxation schemes for \eqref{2DCNS3} in the FV framework. Suppose that the computational domain $\Omega$ is divided into rectangular meshes $\Omega=\cup_{i,j}\Omega_{ij}$, where $\Omega_{ij}=(x_{i-1/2}, x_{i+1/2})\times(y_{j-1/2},y_{j+1/2})$ with $(x_i,y_j)$ as the center, $i=0,1,\dots,I$ and $j=0,1,\dots,J$. Then, a second-order LW type relaxation scheme for \eqref{2DCNS3} is proposed as follows. 

Suppose that the data at time $t=t_n$ are piecewise linear, i.e., 
\begin{equation}
	\begin{aligned}		
		&\bw(x,y,t_n)=\bw^n_{i,j}+(\pt_x \bw)^n_{i,j}(x-x_j)+(\pt_y \bw)^n_{i,j}(y-y_i),\\
		&\bfv_X(x,y,t_n)=(\bfv_X)^n_{i,j}+(\pt_x \bfv_X)^n_{i,j}(x-x_j),\\
		&\bfv_Y(x,y,t_n)=(\bfv_Y)^n_{i,j}+(\pt_y \bfv_Y)^n_{i,j}(y-y_i),\\
		& \tbw(x,y,t_n)=\tbw(\bw(x,y,t_n)),	
		\quad\quad   \forall(x,y)\in\Omega_{ij},
	\end{aligned}
\end{equation}
where $\bfU^n_{i,j}$ is the average value of $\bfU$ over the cell $\Omega_{ij}$ at time $t=t_n$, $(\pt_x \bw)^n_{i,j}$, $(\pt_x \bfv_X)^n_{i,j}$, $(\pt_y \bw)^n_{i,j}$ and $(\pt_y \bfv_Y)^n_{i,j}$ stand for the slopes.

Then, a second-order LW type scheme for \eqref{2DCNS4} can be written as
\begin{equation}
	\begin{aligned}
		\bfU^{n+1}_{i,j}=&\bfU^{n}_{i,j}-\frac{\Delta t}{\Delta x_i}		\big(\bF^{n+1/2}_{i+1/2,j}-\bF^{n+1/2}_{i-1/2,j}\big)
		-\frac{\Delta t}{\Delta y_j}
		\big(\bG^{n+1/2}_{i,j+1/2}-\bG^{n+1/2}_{i,j-1/2}\big)\\
		&-\Delta t\bh^{n+1/2}_{i,j}
		+\frac{\Delta t}{2}\big(\bH^n_{i,j}+\bH^{n+1}_{i,j}\big).\label{2Devolution}
	\end{aligned}  
\end{equation} 
We use the notations that $\Delta x_i=x_{i+1/2}-x_{i-1/2}$, $\Delta y_j=y_{j+1/2}-y_{j-1/2}$, and 
\begin{equation}
	\begin{aligned}
		&\bF^{n+1/2}_{i+1/2,j}=\bF\big(\left(\bfU^X\right)^{n+1/2}_{i+1/2,j}\big),
		\quad
		\bG^{n+1/2}_{i,j+1/2}=\bG\big(\left(\bfU^Y\right)^{n+1/2}_{i,j+1/2}\big),
		\\		
		&	\bH^n_{i,j}=\bH\left(\bfU^n_{i,j}\right),\quad
		\bh^{n+1/2}_{i,j}=(\mathbf{0},(\bh_{\bfv_X})^{n+1/2}_{i,j},(\bh_{\bfv_Y})^{n+1/2}_{i,j})^\top,
	\end{aligned} \label{2Dflux}
\end{equation}
where
\begin{equation}
	\begin{aligned}
		&(\bh_{\bfv_X})^{n+1/2}_{i,j}
		=\bh\big(\frac{1}{2}(\tbw^{n+1/2}_{i,j+1/2}+\tbw^{n+1/2}_{i,j-1/2}),
		\frac{1}{\Delta y_j}(\tbw^{n+1/2}_{i,j+1/2}-\tbw^{n+1/2}_{i,j-1/2}) \big),\\
		&(\bh_{\bfv_Y})^{n+1/2}_{i,j}
		=\bh\big(\frac{1}{2}(\tbw^{n+1/2}_{i+1/2,j}+\tbw^{n+1/2}_{i-1/2,j}),
		\frac{1}{\Delta x_i}(\tbw^{n+1/2}_{i+1/2,j}-\tbw^{n+1/2}_{i-1/2,j}) \big),\\
		&\tbw^{n+1/2}_{i,j+1/2}:=\tbw(\bw^{n+1/2}_{i,j+1/2}).	
	\end{aligned} 
\end{equation}
Note that it suffices to determine the mid-point values $(\bfU^X)^{n+1/2}_{i+1/2,j}$ and $(\bfU^Y)^{n+1/2}_{i,j+1/2}$. For brevity, in the following we present only the evaluation of $(\bfU^X)^{n+1/2}_{i+1/2,j}$. The value of $(\bfU^Y)^{n+1/2}_{i,j+1/2}$ can be obtained similarly.

The mid-point value $(\bfU^X)^{n+1/2}_{j+1/2,i}$ can be computed by the formula
\begin{equation}
	\big(\bfU^X\big)^{n+1/2}_{i+1/2,j}
	=\big(\bfU^X\big)^{n}_{i+1/2,j}
	+\frac{\Delta t}{2}\big(\pt_t \bfU^X\big)^{n+1/4}_{i+1/2,j}.	\label{2Dmid-point}
\end{equation}  

To obtain $\left(\bfU^X\right)^{n}_{i+1/2,j}$, at each point $(x_{i+1/2},y_j,t_n)$, we solve the 1-D RP
\begin{equation}
	\left.\begin{cases}	
		\begin{aligned}
			&\pt_t\bfU^X+\pt_x \bF^X(\bfU^X)=0, \\
			&\bfU^X(x,y_j,0)=\left.\begin{cases}
				(\bfU^X)^{n,l}(x_{i+1/2}-0,y_j),\,\,\,\,\,\quad x<x_{i+1/2},\\
				(\bfU^X)^{n,r}(x_{i+1/2}+0,y_j),\,\,\,\,\,\quad x>x_{i+1/2},
			\end{cases}	
			\right.
		\end{aligned}
	\end{cases}	\right.\label{2DRP}
\end{equation}
by fixing a $y$-coordinate for $y=y_j$, where $\bF^X(\bfU^X)=(\bfv_X,\bF_{\bfv_X}(\bw))^\top$, $(\bfU^X)^{n,l}(x,y)$ and $(\bfU^X)^{m,r}(x,y)$ are two linear polynomials defined on the two neighboring computational cells, respectively. The RP \eqref{2DRP} can be solved by approximate Riemann solvers, such as a Roe-type solver \cite{Roe1981} or HLLC-type solver \cite{BCLC1997}. 

To determine $(\pt_t\bfU^X)^{n}_{i+1/2,j}$, we need to solve the so-called quasi 1-D GRP
\begin{equation}
	\left.\begin{cases}	
		\begin{aligned}
			&\pt_t\bfU^X+\pt_x \bF^X(\bfU^X)+\pt_y\bG^X(\bfU^Y)+\bh^X(\tbw,\pt_y \tbw)=\bH^X(\bfU^X), \\
			&\bfU^X(x,y,0)=\left.\begin{cases}
				(\bfU^X)^{n,l}(x,y),\,\,\,\,\,\quad x<x_{i+1/2},\\
				(\bfU^X)^{n,r}(x,y),\,\,\,\,\,\quad x>x_{i+1/2},
			\end{cases}	
			\right.
		\end{aligned}
	\end{cases}	\right.\label{2DGRP0}
\end{equation}
where 
\begin{equation}
	\begin{aligned}
		&\bG^X(\bfU^Y)=(\bfv_Y,\bG_{\bfv_X}(\bw))^\top,\,\,\,\,\,
		\bh^X(\tbw,\pt_y \tbw)=(\mathbf{0},\bh_{\bfv_X}(\tbw,\pt_y \tbw))^\top,\\
		&\bH^X(\bfU^X)=(\mathbf{0},\frac{1}{\epsilon}(\bff_X(\bw)-\bfv_X))^\top.
	\end{aligned}
\end{equation}
Note that $\bG^X(\bfU^Y)$ is linear in $\bfv_Y$. Thus, we can linearize the convection term in \eqref{2DGRP0} around the approximate local Riemann solution $(\bfU^X)^{n}_{i+1/2,j}$ to derive 
\begin{equation}
	\left.\begin{cases}	
		\begin{aligned}
			&\pt_t\bfU^X+M^X\pt_x \bfU^X=-M^Y\pt_y\bfU^Y+\bH^X(\bfU^X), \\
			&\bfU^X(x,y_j,0)=\left.\begin{cases}
				(\bfU^X)^l(x,y_j),\,\,\,\,\,\quad x<x_{i+1/2},\\
				(\bfU^X)^r(x,y_j),\,\,\,\,\,\quad x>x_{i+1/2},
			\end{cases}	
			\right.
		\end{aligned}
	\end{cases}	\right.\label{2DGRP}
\end{equation}
where the $y$-coordinate is fixed for $y=y_j$, and
\begin{equation*}
	\begin{small}
		M^X=\begin{pmatrix}
			\mathbf{0}\,&\bI_4\,\\
			\pt_\bw\bF_{\bfv_X}(\bw^{n}_{i+1/2,j})\,&\mathbf{0}\,\\
		\end{pmatrix},
		\,\,\,
		M^Y=\begin{pmatrix}
			\mathbf{0} \, &\bI_4\,\\
			\pt_\bw\bG_{\bfv_X}(\bw^{n}_{i+1/2,j})+\pt_\bw\bh_{\bfv_X}(\bw^{n}_{i+1/2,j})\,
			&\mathbf{0}\,\\
		\end{pmatrix}.
	\end{small}
\end{equation*}
Then, based on the 2-D acoustic GRP solver \cite{LD2016}, the 2-D IMEX-GRP solver for \eqref{2DGRP} takes the form 
\begin{equation*}
	\begin{aligned}
		(\pt_t\bfU^X)^{n+1/4}_{i+1/2,j}
		=&-R\Lambda^+R^{-1}(\pt_x\bfU^X)^n_{i,j}
		-RI^+R^{-1}M^Y(\pt_y\bfU^Y)^n_{i,j}+H\big((\bfU^X)^{n+1/2}_{i+1/2,j}\big)\\
		&-R\Lambda^-R^{-1}(\pt_x\bfU^X)^n_{i+1,j}-RI^-R^{-1}M^Y(\pt_y\bfU^Y)^n_{i+1,j}
		, \label{2DNSIMEXGRP}
	\end{aligned}
\end{equation*}
where $\Lambda =\diag(\lambda_1,\dots,\lambda_8)$, $\lambda_k$, $k= 1,\dots, 8$, are the eigenvalues of $M^X$, $R$ is the right eigenmatrix of $M^X$,  $\Lambda^+=\diag(\max(\lambda_k,0))$, $\Lambda^-=\diag(\min(\lambda_k,0))$, 
$I^+=\frac{1}{2}\diag(1+\sign(\lambda_k))$ and $I^-=\frac{1}{2}\diag(1-\sign(\lambda_k))$.

Finally, the slopes can be updated in analogy with the 1-D case, see \eqref{1Dslope}. 

\begin{remark}
	For \eqref{2DGRP}, since we just want to construct the flux normal to cell interfaces, the tangential effect is regarded as a source. That is, we solve the 1-D GRP at a point $(x_{i+1/2}, y_j)$ on the interface by considering the effect tangential to the interface $x=x_{i+1/2}$. As is pointed out in \cite{LD2016,LL2019}, it is different from the 1-D version in that the multidimensional effect is included.
\end{remark}

\section{Numerical examples} \label{sectest}
Several benchmark problems are tested to illustrate the performance of current schemes for the 1-D viscous Burgers, and 1-D and 2-D Navier-Stokes system.

For the schemes proposed in \Cref{example4} and \Cref{example5}, the approximate local Riemann solutions $U^{n}_{j+1/2}$ and $(\bfU^X)^{n}_{i+1/2,j}$ are obtained by Roe-type solvers as in \cite{Roe1981}. The details for the construction of the 1-D Roe-type solver can be found in the \Cref{ROE}, and the 2-D case can be addressed similar to the 1-D case with some non-essential adaption. The numerical treatment of boundary conditions for the current relaxation schemes can be found in \Cref{boundarytreatment}.

\begin{test}[Smooth problem]
	The first test is the initial-boundary value problem (IBVP) of the 1-D viscous Burgers equation studied by Asaithambi \cite{A2010}:
	\begin{equation}
		u(x,0)=\frac{2\mu\pi\sin(\pi x)}{2+\cos(\pi x)},\,\,\,\mathrm{and}
		\,\,\,\,\, u(0,t)=0,\,\,u(1,t)=0,\,\,\,\,\, x\in[0,1],\, t>0.
	\end{equation}
	The exact solution is 
	\begin{equation}
		u(x,t)=\frac{2\mu\pi \sin(\pi x)\exp(-\pi^2\mu t)}{2+\exp(-\pi^2\mu t)\cos(\pi x)}.
	\end{equation}
	The purpose of this test is to verify the order of accuracy of the scheme \eqref{exampleB0}--\eqref{exampleBend}. In this test, a uniform mesh with $N$ grid points is used. Numerical simulations are conducted for $\epsilon=((\Delta x)^2/\mu)^k$ with $k=1$ and $k=0.9$. Other parameters are set as $\mu=0.01$ and $a=0$.  The CFL number takes a value of $0.7$.
	
	The $L^1$, $L^\infty$-errors and orders at time $t=10.0$ are presented in \cref{Tab:Burgers}, in which ``Nstep'' represents the number of time-stepping. Note that for $k=1$, the expected second-order accuracy can be achieved. For $k=0.9$, there is a gain on the number of time-stepping compared to the case that $k=1$, while the second-order accuracy is lost, since the relaxation model error at order of $\epsilon\approx (\Delta x)^{1.8}$ dominates the numerical error at order of $(\Delta x)^2+\Delta t^2\approx (\Delta x)^{1.9}$. 
	
	\begin{table}[htbp]
		\centering
		\setlength{\abovecaptionskip}{0.0cm}
		\setlength{\belowcaptionskip}{0.0cm}
		\setlength{\tabcolsep}{0.45mm}
		\begin{tabular}{c|ccccc|ccccc}
			\toprule
			& \multicolumn{5}{c|}{$\epsilon=((\Delta x)^2/\mu)^k$,\,\,$k=1$} 
			& \multicolumn{5}{c}{$\epsilon=((\Delta x)^2/\mu)^k$,\,\,$k=0.9$}\\
			\hline
			N  & Nstep  &$L^1$ error&  Order   &$L^\infty$ error  &  Order
			& Nstep  &$L^1$ error&  Order   &$L^\infty$ error  &  Order     \\
			\hline
			64   &586    &4.019e-05 &        &6.437e-05   &           
			&487    &5.934e-05 &        &9.502e-05   &            \\
			128  &2341   &1.017e-05 &1.983   &1.623e-05   &1.988  
			&1814   &1.730e-05 &1.778   &2.762e-05   &1.783      \\
			256  &9363   &2.557e-06 &1.991   &4.077e-06   &1.993    
			&6770   &5.015e-06 &1.786   &7.996e-06   &1.788       \\
			512  &37450  &6.412e-07 &1.996   &1.022e-06   &1.997      
			&25265  &1.449e-06 &1.791   &2.309e-06   &1.792        \\
			1024 &149797 &1.605e-07 &1.998   &2.557e-07   &1.998  
			&94292  &4.181e-07 &1.793   &6.661e-07   &1.794     \\
			\bottomrule
		\end{tabular}
		\caption{\small Accuracy test for the IBVP of 1-D viscous Burgers equation at $t = 10.0$.}     \label{Tab:Burgers}
	\end{table}  
\end{test}
\begin{test}[Manufactured solution]
	In this test, we employ the following manufactured solution to assess the accuracy of the scheme proposed in \Cref{example4}:
	\begin{equation}
		\small
		\rho(x,y,t)=1+0.2\sin(\pi(x+y-2t)),\;\; U(x,y,t)=1,\;\; V(x,y,t)=1,\;\; p(x,y,t)=1. \label{manufacture}
	\end{equation}
	Here, and in the sequel, $U$ and $V$ denote the velocity components in $x$- and $y$- directions, respectively. The computational domain is $\Omega =[0, 2]\times[0, 2]$ with periodic boundary conditions. In the computation, a uniform mesh with $N\times N$ grid points is used. The parameters are set as $\gamma=1.4$, $Pr=0.72$, $\mu=0.001$, $a=1.2$ and $\epsilon=0.1 (\Delta x)^2$ or $\epsilon=\mu/100$. The CFL number takes a value of $0.3$.
	
	The $L^1$- and $L^\infty$-errors and orders for the cell averages of the density $\rho$ at $t=1.0$ with $\epsilon=0.1 (\Delta x)^2$ and $\epsilon=\mu/100$ are presented in \cref{Tab:2DNS}. The numerical error consists of two parts, the discretization error of order $(\Delta x)^2$ and the relaxation error of order $\ep$. When keeping $\ep=O((\Delta x)^2)$, the second-order accuracy can be achieved. When fixing $\ep$ as a constant, the second-order accuracy is lost for small $\Delta x$. 
	
	\begin{table}[htbp]
		\centering
		\setlength{\abovecaptionskip}{0.0cm}
		\setlength{\belowcaptionskip}{0.0cm}
		\setlength{\tabcolsep}{1.5mm}
		\begin{tabular}{c|cc|cc|cc|cc}
			\toprule
			\footnotesize
			& \multicolumn{4}{c|}{$\epsilon=0.1 (\Delta x)^2$}
			& \multicolumn{4}{c}{$\epsilon=\mu/100$}\\
			\hline
			N &$L^1$ error&  Order   &$L^\infty$ error  &  Order
			&$L^1$ error&  Order   &$L^\infty$ error  &  Order     \\
			\hline
			16   &3.797e-03 &        &7.307e-03   &           
			&2.997e-03 &        &5.616e-03   &            \\
			32   &8.312e-04 &2.191   &1.426e-03   &2.357   
			&7.807e-04 &1.931   &1.417e-04   &1.987      \\
			64   &1.993e-04 &2.060   &3.463e-04   &2.042   
			&1.928e-04 &2.018   &3.445e-04   &2.040       \\
			128  &5.033e-05 &1.986   &8.716e-05   &1.990     
			&4.621e-05 &2.061   &7.948e-05   &2.116       \\
			256  &1.276e-04 &1.980   &2.199e-05   &1.986     
			&1.594e-05 &1.536   &2.761e-05   &1.525       \\
			\bottomrule
		\end{tabular}   
		\caption{\small Accuracy test for the manufactured solution \eqref{manufacture}: the $L^1$- and $L^\infty$-errors and orders for the density $\rho$ at $t=1.0$.}   \label{Tab:2DNS}  
	\end{table}  
\end{test}
\begin{test}[Sod problem]  \label{Sod}
	In this test, we apply the scheme proposed in \Cref{example4} to the shock tube problem by Sod \cite{Sod1978}. The computational domain is $\Omega=[0, 1]$ with non-reflecting boundary condition on both ends. The initial condition is given by 
	\begin{equation}
		(\rho,u,p)=\begin{cases}
			\begin{aligned}
				&(1, 0, 1),\,\,\,\,\, &0 < x < 0.5,\\
				&(0.125, 0, 0.1), \,\,\,\,\, &0.5 < x < 1.
			\end{aligned}
		\end{cases}
	\end{equation}
	Numerical simulations are conducted for three different viscosity coefficients, i.e. $\mu=0.002$, $0.0005$ and $0.00005$. Other parameters are set as $\gamma=1.4$, $Pr=0.72$, $\epsilon=\mu/100$ and $a=1.0$. A uniform mesh with $200$ grid points is adopted. The CFL number takes a value of $0.7$. The minmod limiter \footnote{In order to suppress the local oscillations near discontinuities, we apply the minmod limiter (see \cite{BLW2006,vanLeer1979}) to $(U_x)^{n+1}_j$ in \eqref{1Dslope},
		\begin{align}
			\mathrm{minmod}\left(\alpha\frac{U^{n+1}_{j+1}-U^{n+1}_j}{\Delta x},
			(U_x)^{n+1}_j,\alpha\frac{U^{n+1}_j-U^{n+1}_{j-1}}{\Delta x}\right),	           \label{minmod}
		\end{align}
		where the parameter $\alpha\in[0,2]$.} with limiter parameter $\alpha=2.0$ is employed. 
	
	Numerical solutions at $t=0.2$ are shown in \cref{fig:RP}, where the reference is the exact solution of the Euler equations. From the figures we observe that the numerical solutions of the Navier-Stokes system is getting closer to the exact solution of the Euler system when $\mu$ becomes smaller, as expected.  
	\begin{figure}[htbp]
		\centering
		\setcaptionmargin{0cm}
		\setlength{\abovecaptionskip}{0cm}
		\setlength{\belowcaptionskip}{0cm}
		\subfigure{	\includegraphics[width=0.31\linewidth]{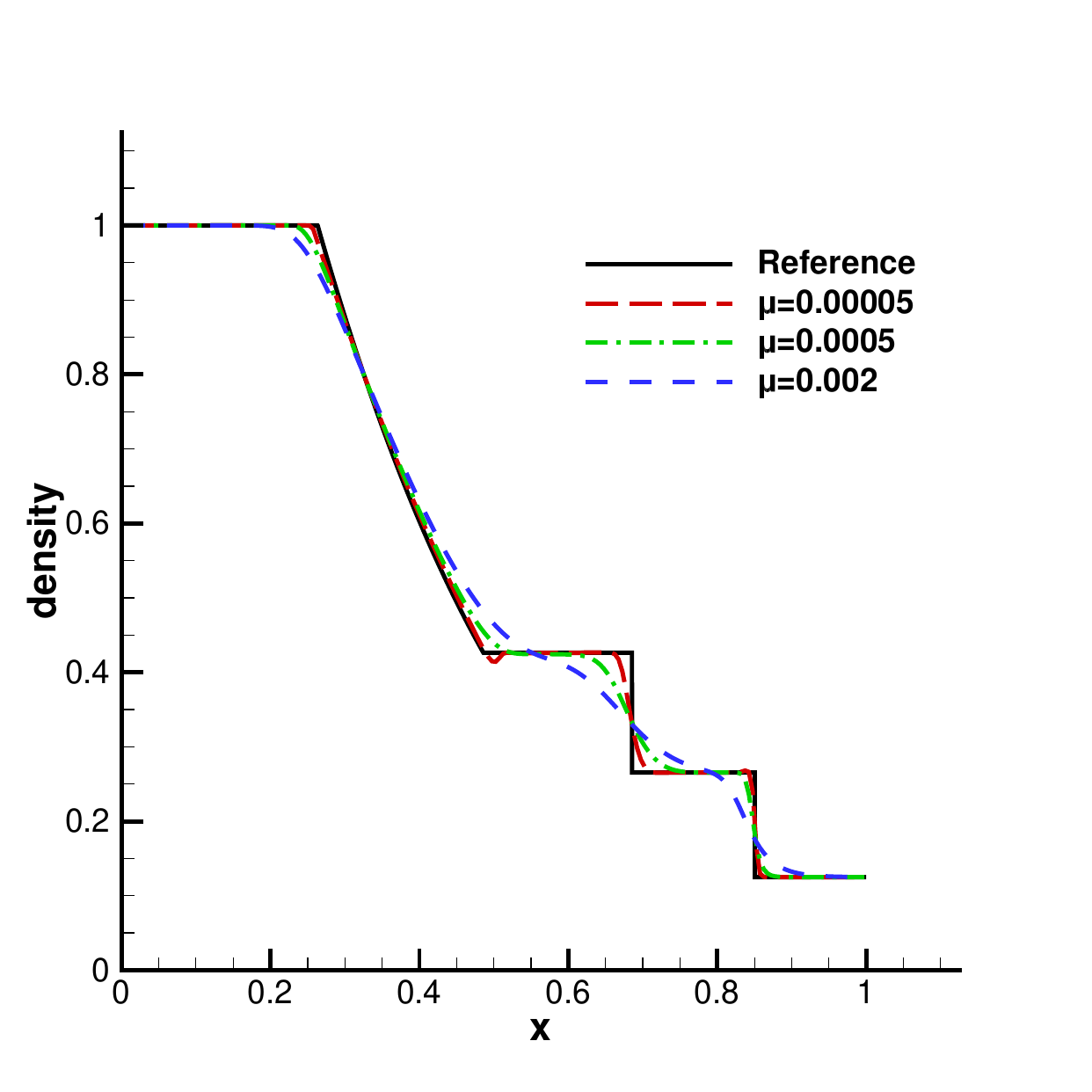}}
		\hspace{0mm}
		\subfigure{	\includegraphics[width=0.31\linewidth]{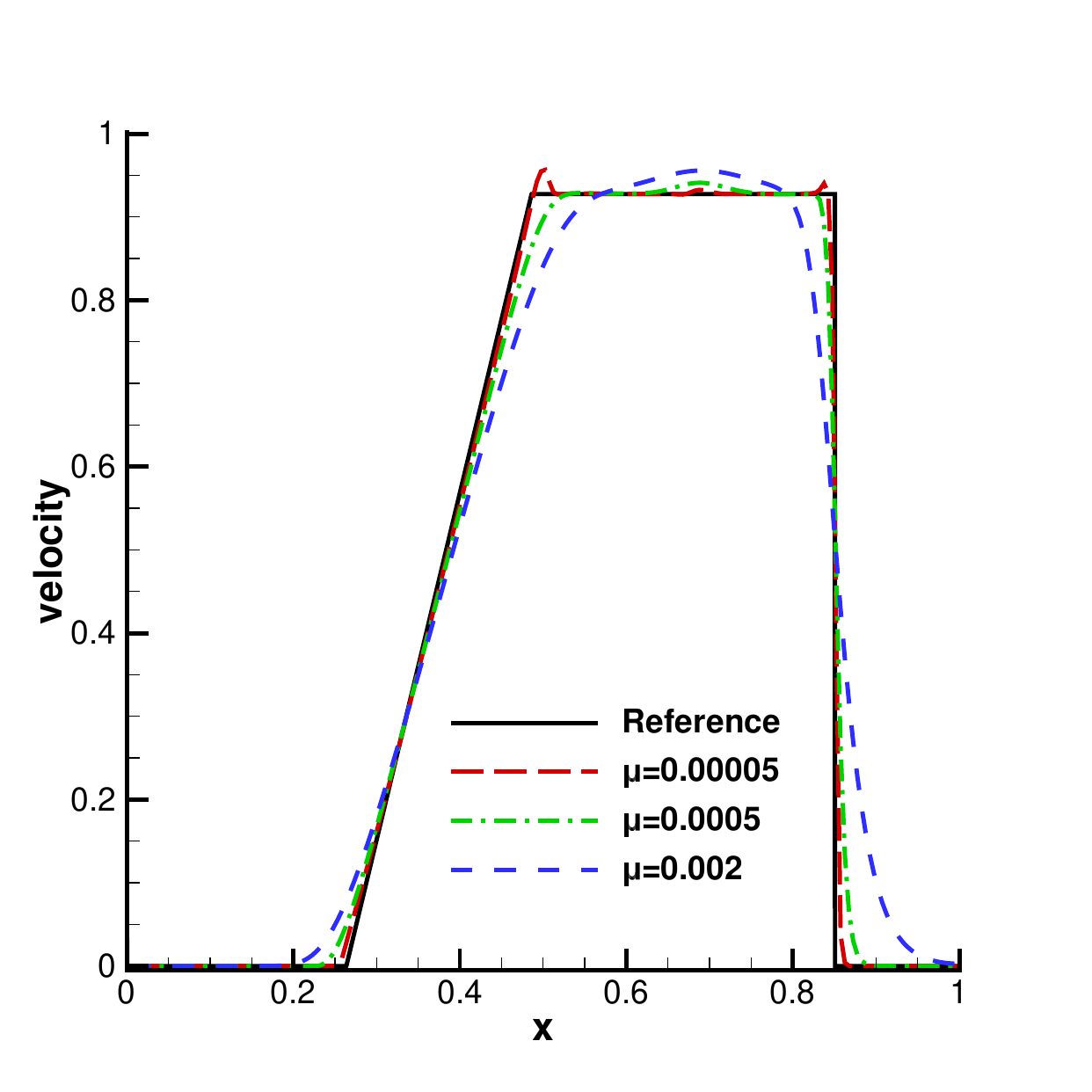}}
		\hspace{0mm}
		\subfigure{	\includegraphics[width=0.31\linewidth]{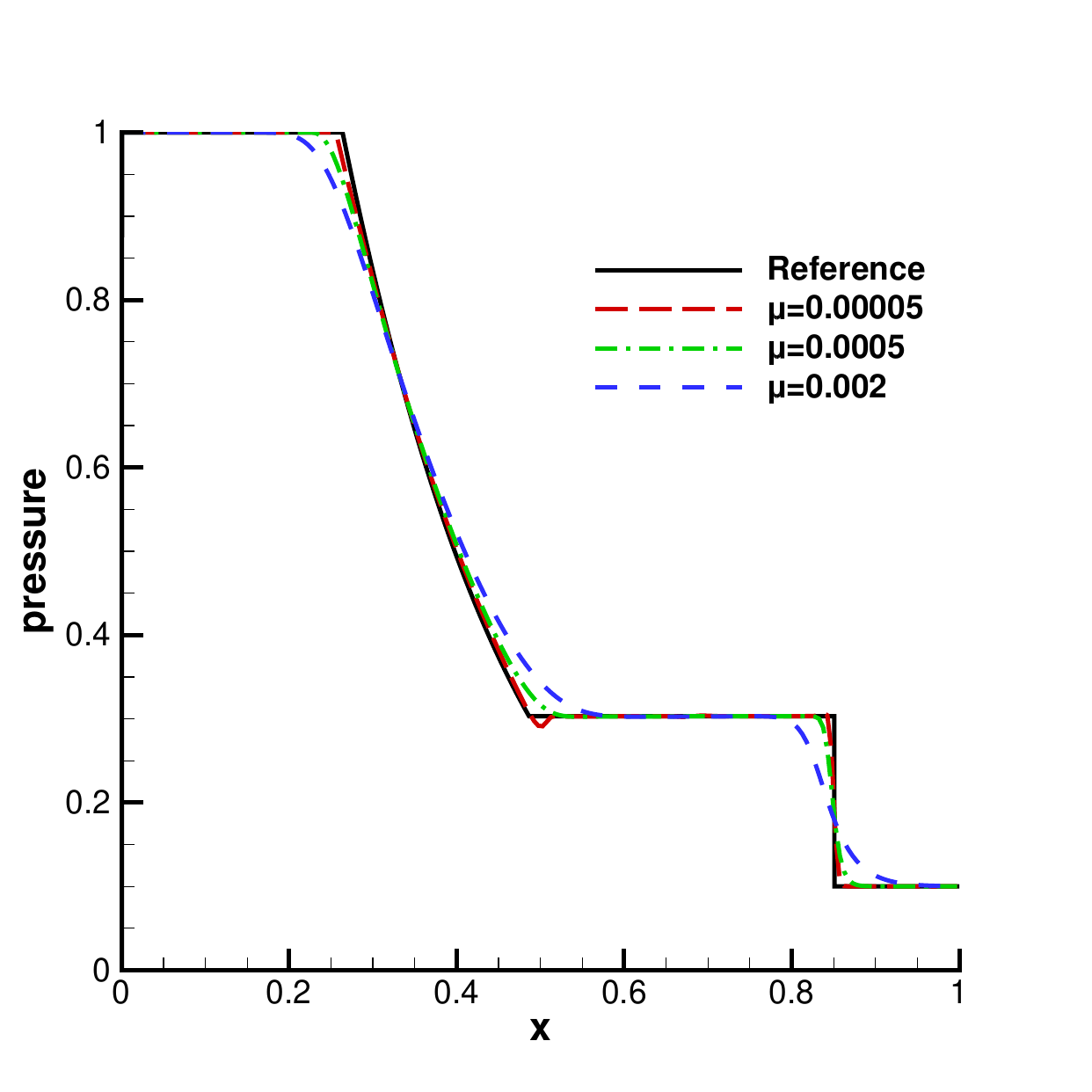}}
		\caption{\small Sod problem: numerical solutions at $t=0.2$ with $200$ grid points for different viscosity coefficients. The reference is the exact solution of the Euler equations.}
		\label{fig:RP}
	\end{figure}
\end{test}
\begin{test}[Laminar boundary layer]  \label{LBL}
	In this test, we apply the scheme proposed in \Cref{example5} to a laminar compressible flow over a flat plate with length $L=100$. The Mach number of the free-stream is $Ma=0.15$ and the Reynolds number is $Re=10^5$ which is based on the free-stream reference conditions and the plate length. As shown in \cref{fig:mesh}, the rectangular computation mesh contains $90\times45$ elements and the flat plate starts at point $(0, 0)$. The non-slip and adiabatic boundary condition is used at the plate and a symmetric boundary condition is imposed at the bottom boundary before the flat plate. The inflow boundary condition is adopted for the inlet plane. Moreover, the outflow boundary condition is enforced at the top and exit planes. The initial state is given by the free-stream. Other parameters are set as $\gamma=1.4$, $Pr=0.72$, $\epsilon=\mu/100$ and $a=1.5$. The CFL number takes a value of $0.3$.
	
	Note that for this free-stream Mach number, flow is nearly incompressible so that incompressible laminar boundary layer theory can be used to assess the performance of our method. Similar tests can be found in \cite{DY2022,PXLL2016}. We have computed till $t=150$, when the flow has approached a steady state. The non-dimensional $U$ and $V$ velocity components at different locations are presented in \cref{fig:Laminar}, where the solid lines are the exact Blasius solutions of $U$ and $V$. In all locations, the numerical solutions match with the exact Blasius solution well. At the upstream location, the boundary layer profile can be accurately captured with only seven grid points within the layer.
	\begin{figure}[htbp]
		\centering
		\setlength{\abovecaptionskip}{0.0cm}
		\setlength{\belowcaptionskip}{0.0cm}
		\subfigure{	\includegraphics[width=0.5\linewidth]{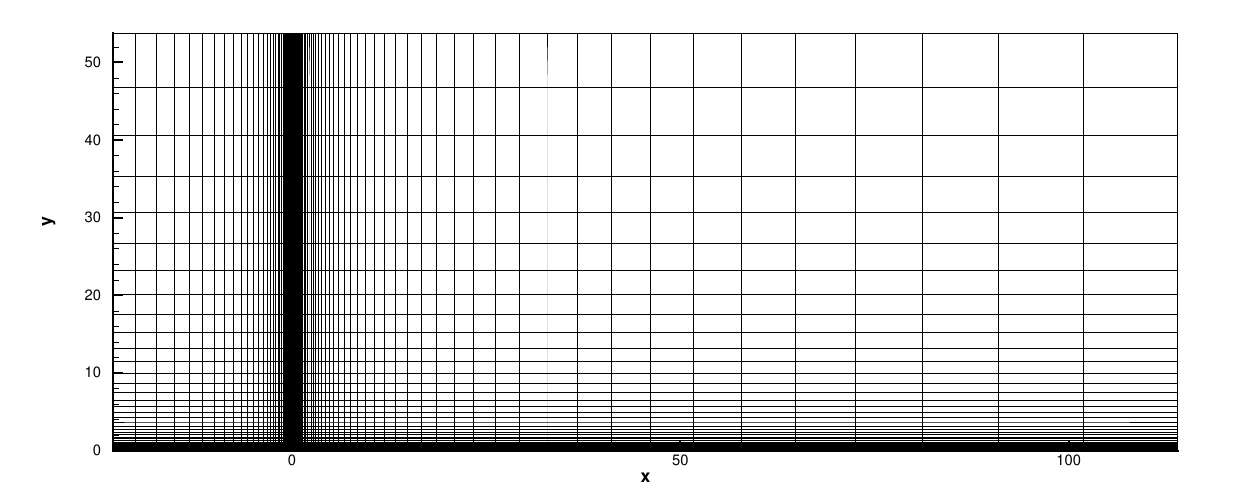}}
		\caption{\small Laminar boundary layer: rectangular meshes.}
		\label{fig:mesh}
	\end{figure}
	\begin{figure}[htbp]
		\centering
		\setlength{\abovecaptionskip}{0.0cm}
		\setlength{\belowcaptionskip}{0.0cm}
		\subfigure{	\includegraphics[width=0.45\linewidth]{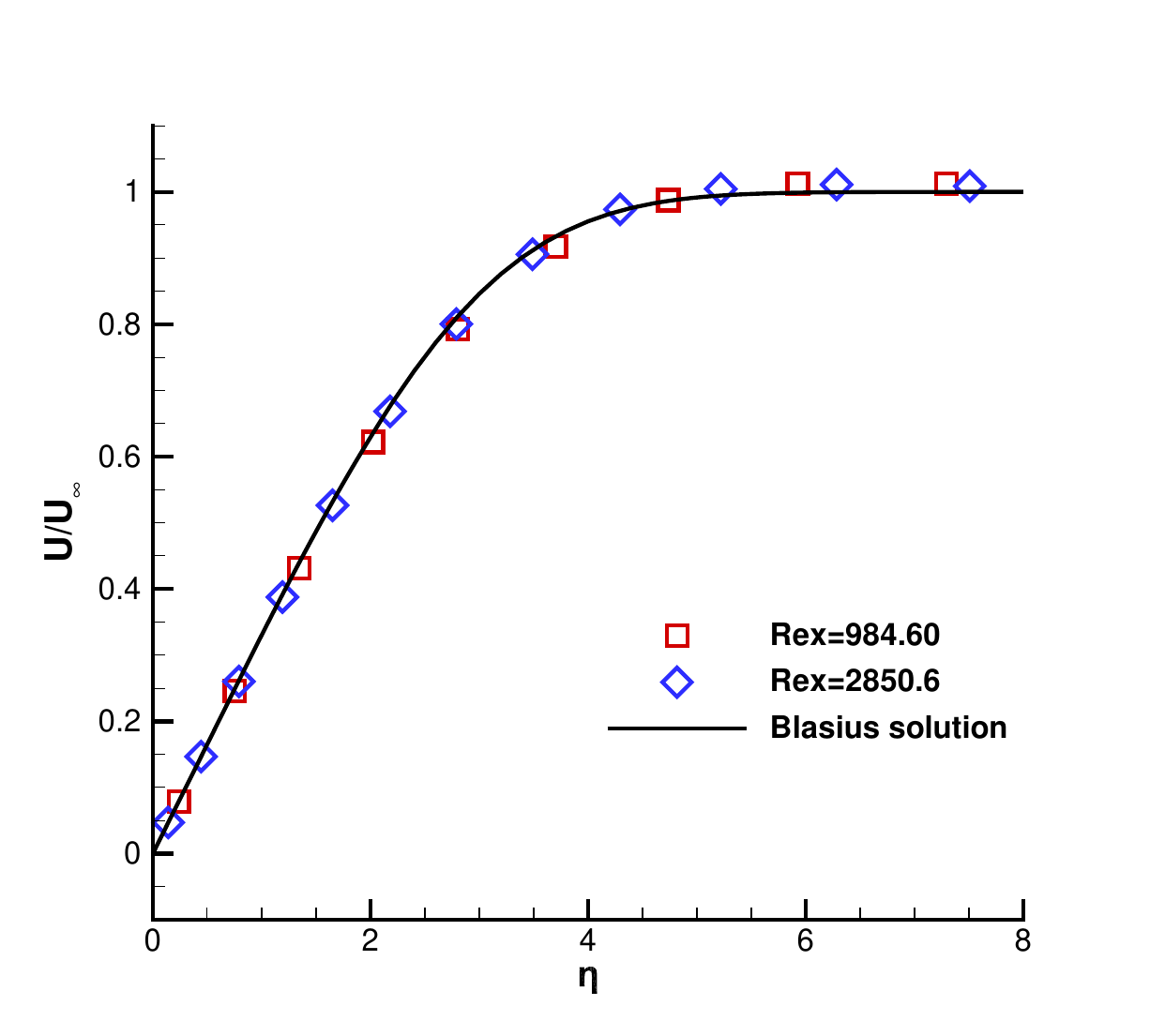}}
		\hspace{1mm}
		\subfigure{	\includegraphics[width=0.45\linewidth]{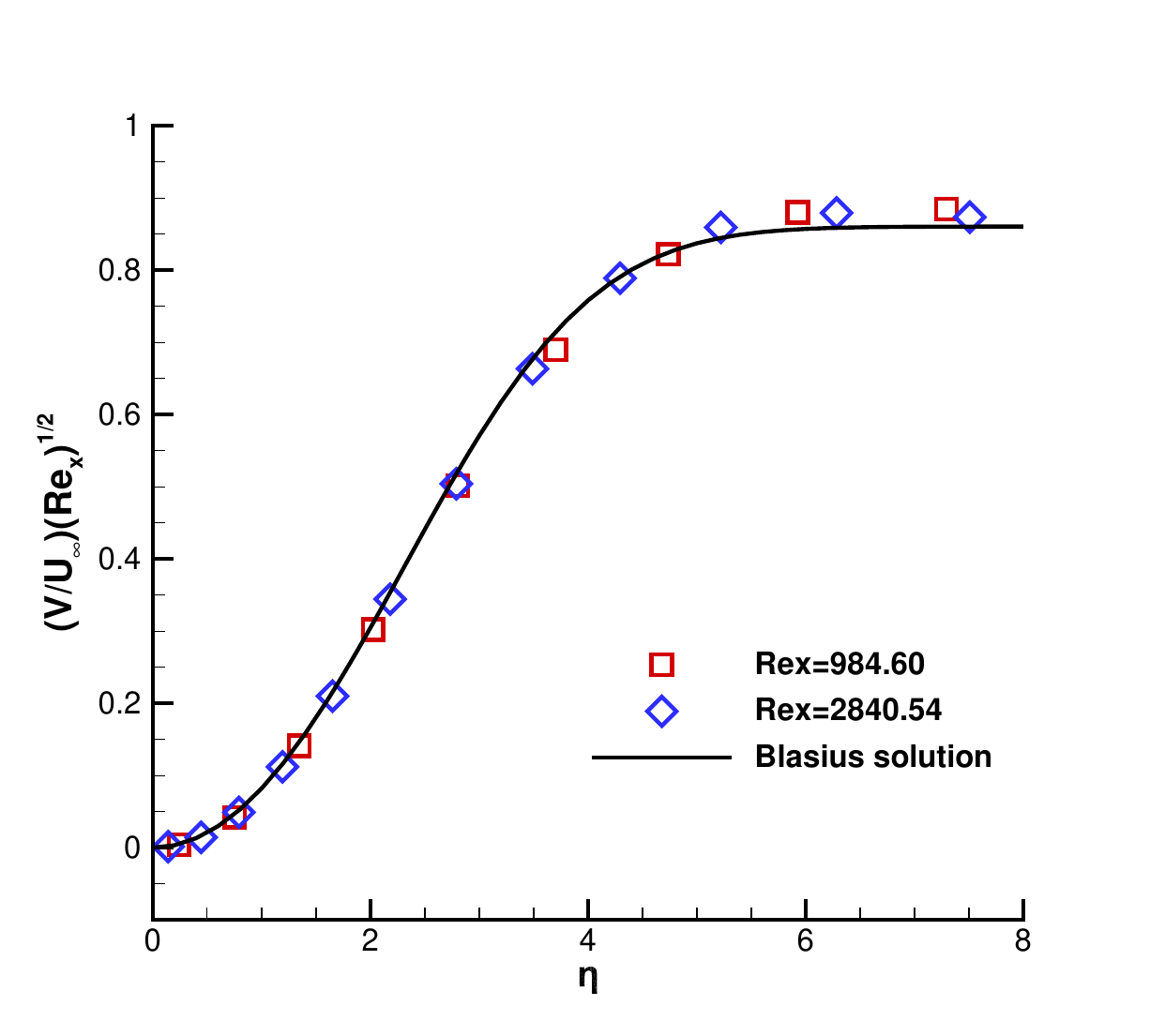}}
		\caption{\small Laminar boundary layer: the non-dimensional $U$ and $V$ velocity components at different locations for $Re=10^5$. $\eta=\frac{y}{x}\sqrt{Re_x}$ is the boundary-layer coordinate, and $Re_x=Re\frac{x}{L}$ is the Reynolds number based on the distance along the plate from the leading edge.}
		\label{fig:Laminar}
	\end{figure}
\end{test}
\begin{test}[Lid-driven cavity flow]  
	In this test, we apply the scheme proposed in \Cref{example5} to the lid-driven cavity flow, which is a typical benchmark for 2-D incompressible or low Mach number viscous flow. The fluid is bounded by a unit square $\Omega=[0, 1]\times[0, 1]$. The upper wall is moving with an horizontal speed $U_{b}=1$. Other walls are fixed. The non-slip and isothermal boundary conditions are applied to all boundaries with the temperature $T_{b}=(\gamma Ma^2)^{-1}$, where the Mach number is chosen to be $Ma=0.15$. The initial flow is stationary with the density $\rho_0=1$ and temperature $T_0=T_b$. Numerical simulations are conducted for two different Reynolds numbers, i.e., $Re = 400$ and $1000$. Other parameters are set as $\gamma=1.4$, $Pr=0.72$, $\epsilon=\mu/100$ and $a=2.0$. The computational domain $\Omega$ is covered uniformly with $85\times85$ mesh points and the CFL number takes a value of $0.3$. On such a coarse mesh, it is challenging to resolve the velocity profiles accurately.
	
	We have computed till $t=50$, when the flow has approached a steady state. The streamlines are shown in \cref{fig:streamline}. The primary flow structures including the primary and secondary vortices are well captured. The results of $U$-velocities along the center vertical line and $V$-velocities along the center horizontal line are shown in \cref{fig:UV}. The simulation results match well with the benchmark data \cite{Ghia1982}. The results demonstrate the accuracy of the current scheme.
	\begin{figure}[htbp]
		\centering
		\setlength{\abovecaptionskip}{0.0cm}
		\setlength{\belowcaptionskip}{0.0cm}
		\subfigure{	\includegraphics[width=0.35\linewidth]{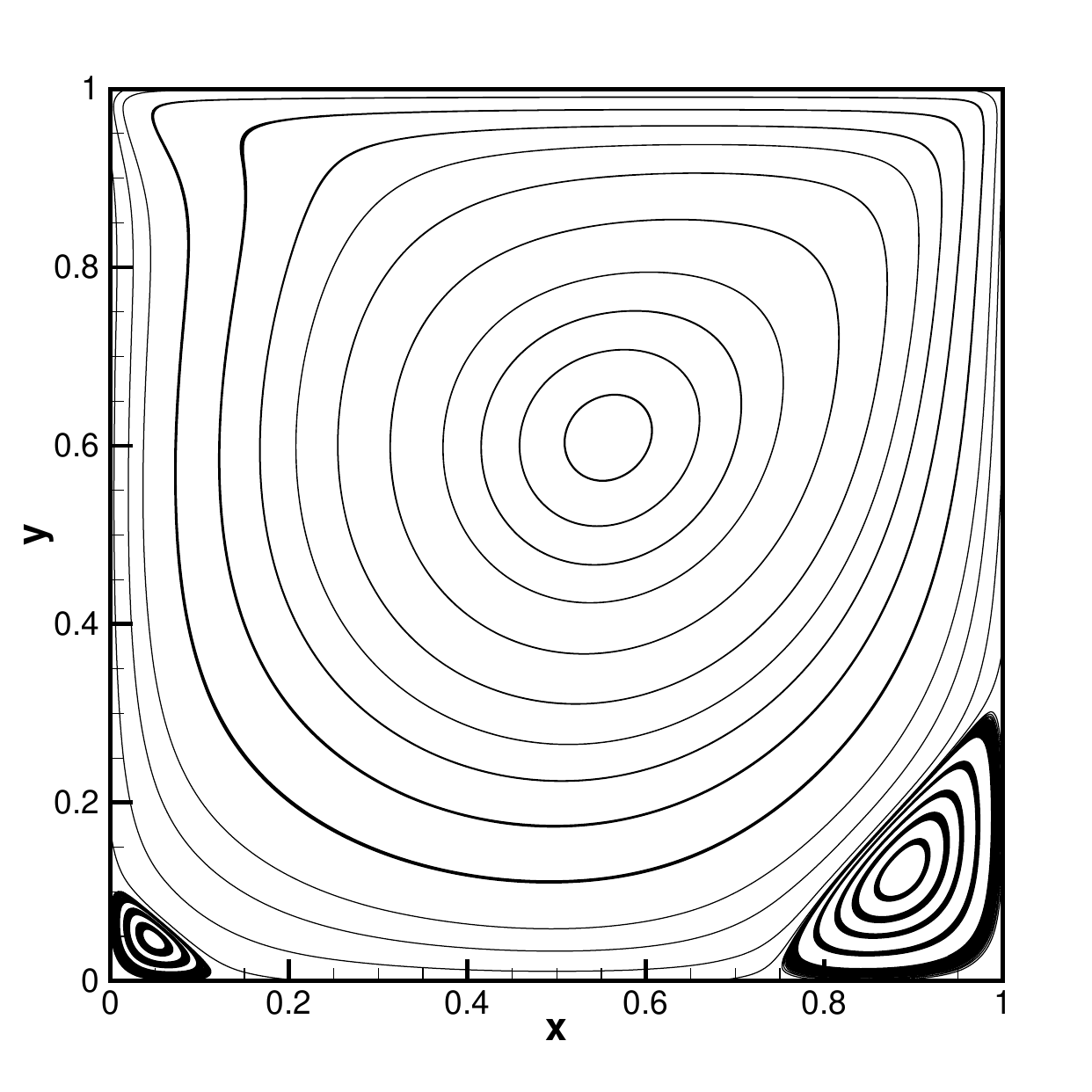}}
		\hspace{10mm}
		\subfigure{	\includegraphics[width=0.35\linewidth]{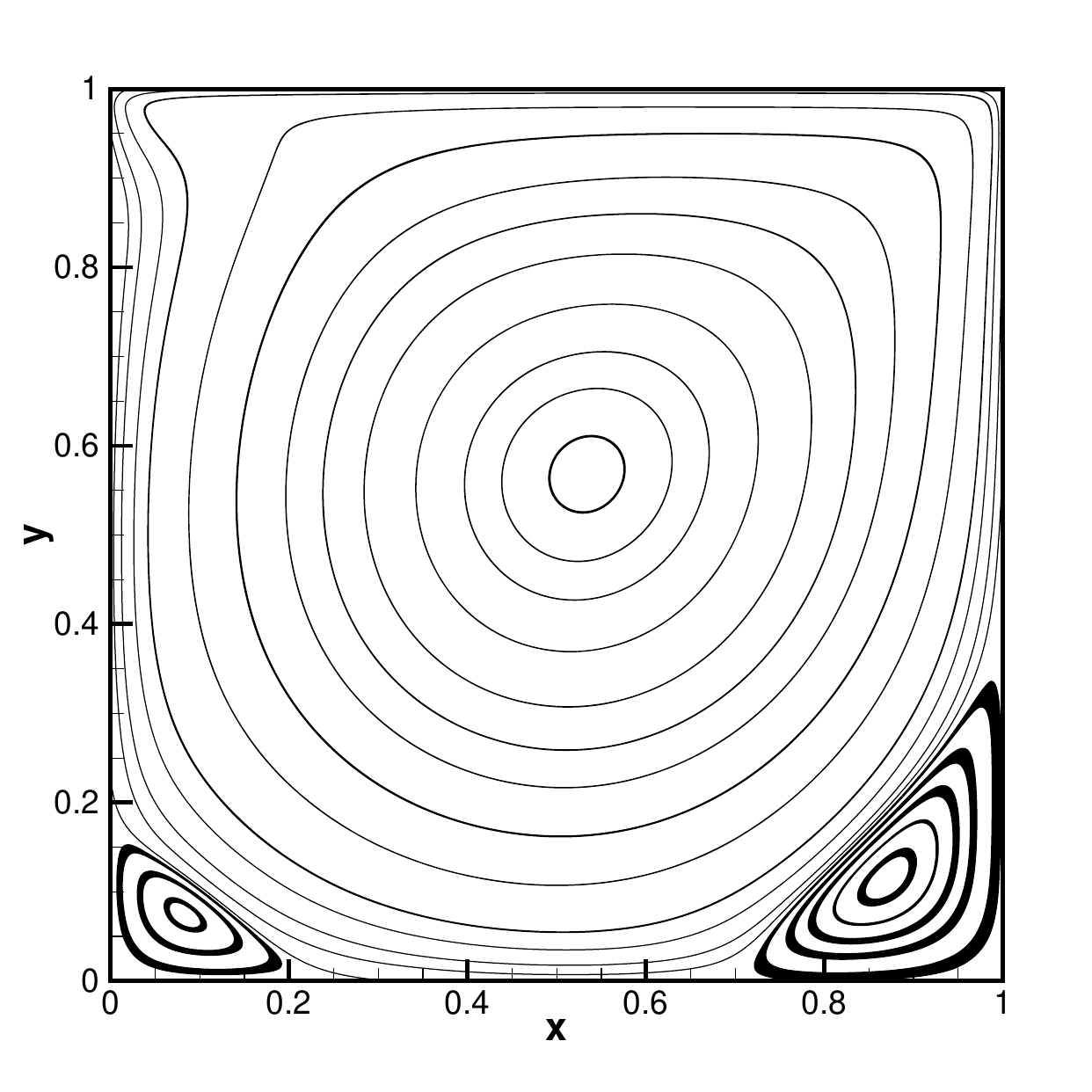}}
		\caption{\small Lid-driven cavity flow: the streamlines at Re = 400 (left) and 1000 (right) with $\Delta x=\Delta y=1/85$.}
		\label{fig:streamline}
	\end{figure}
	\begin{figure}[htbp]
		\centering
		\setlength{\abovecaptionskip}{0.0cm}
		\setlength{\belowcaptionskip}{0.0cm}
		\footnotesize
		\subfigure{	\includegraphics[width=0.40\linewidth]{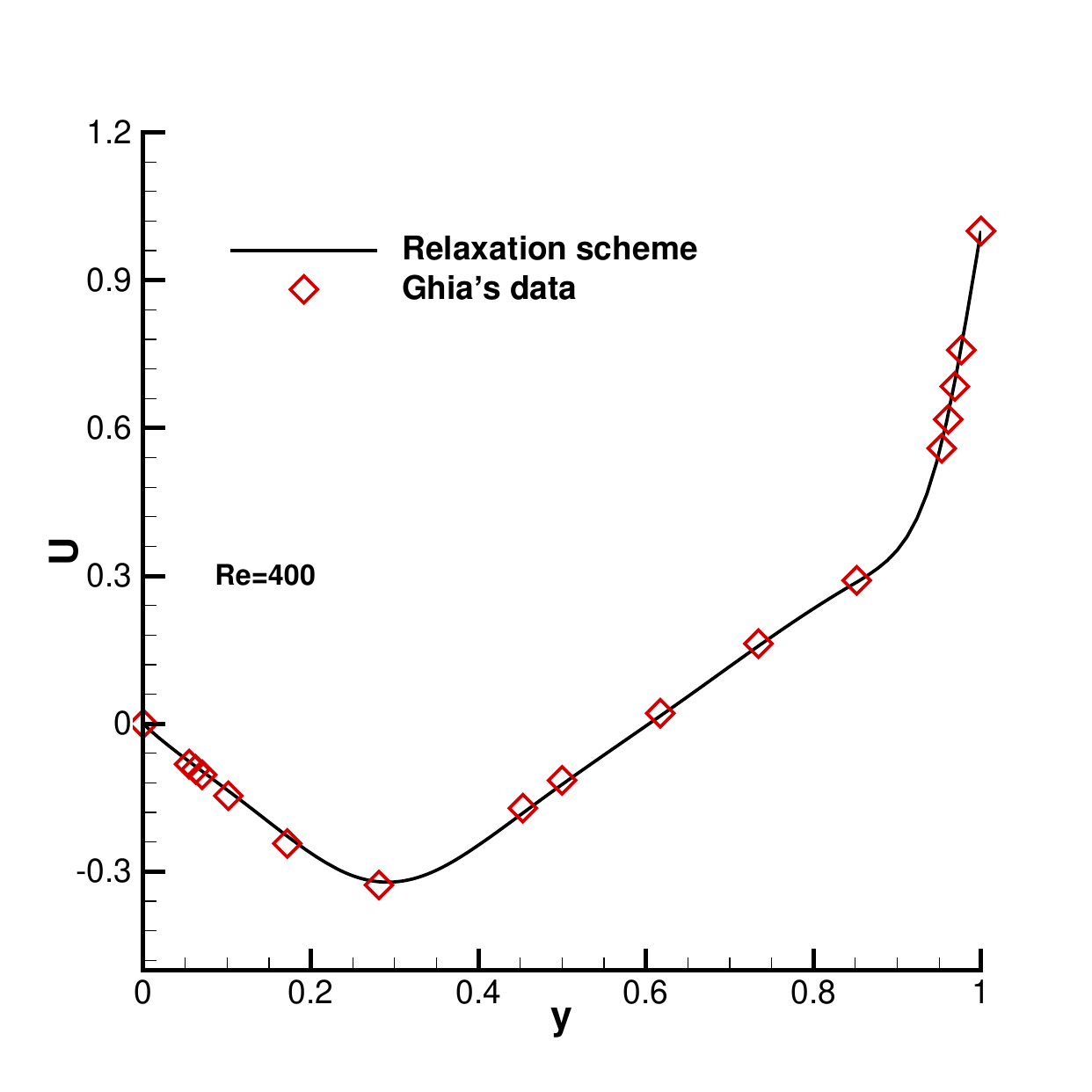}}
		\hspace{5mm}
		\subfigure{	\includegraphics[width=0.40\linewidth]{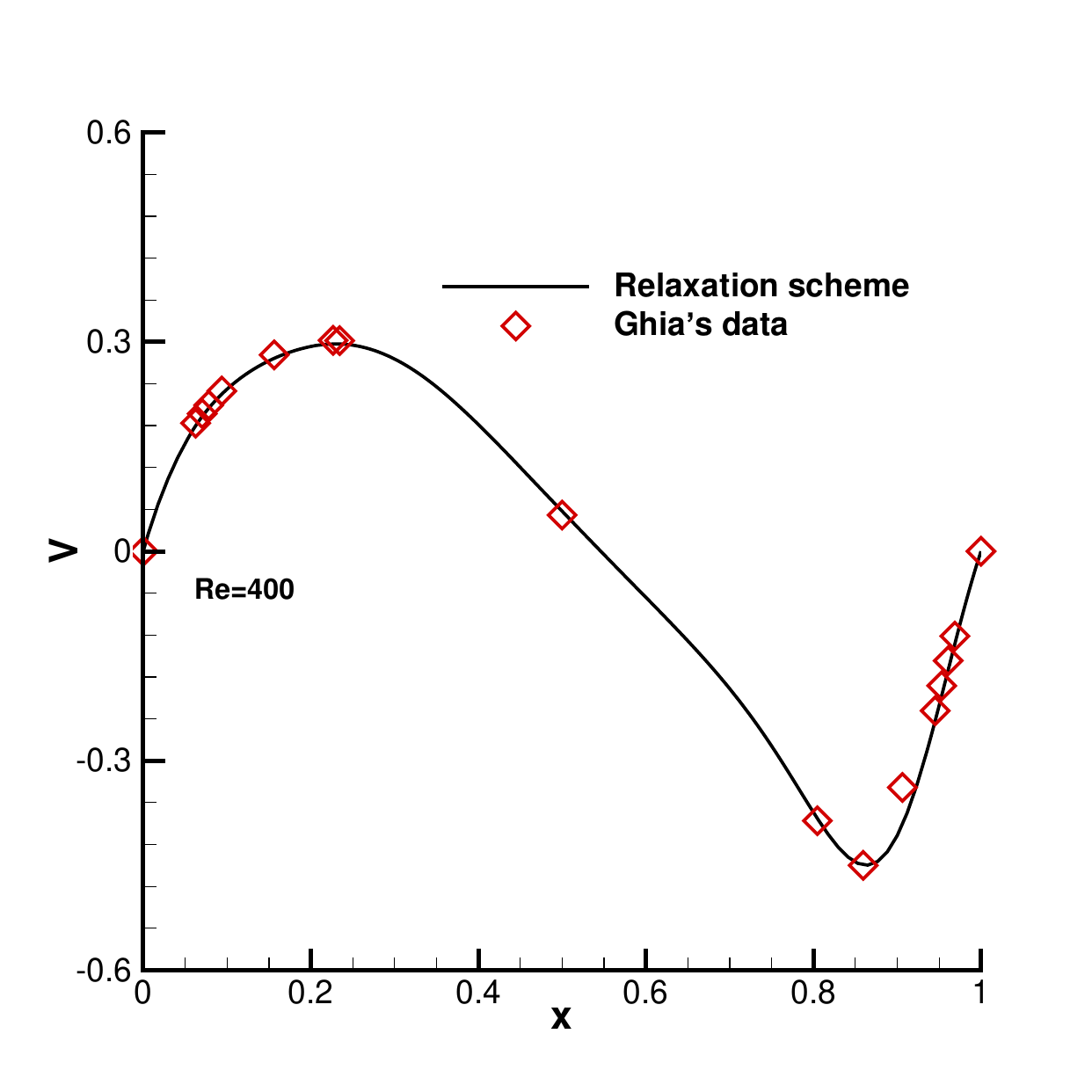}}
		\subfigure{	\includegraphics[width=0.40\linewidth]{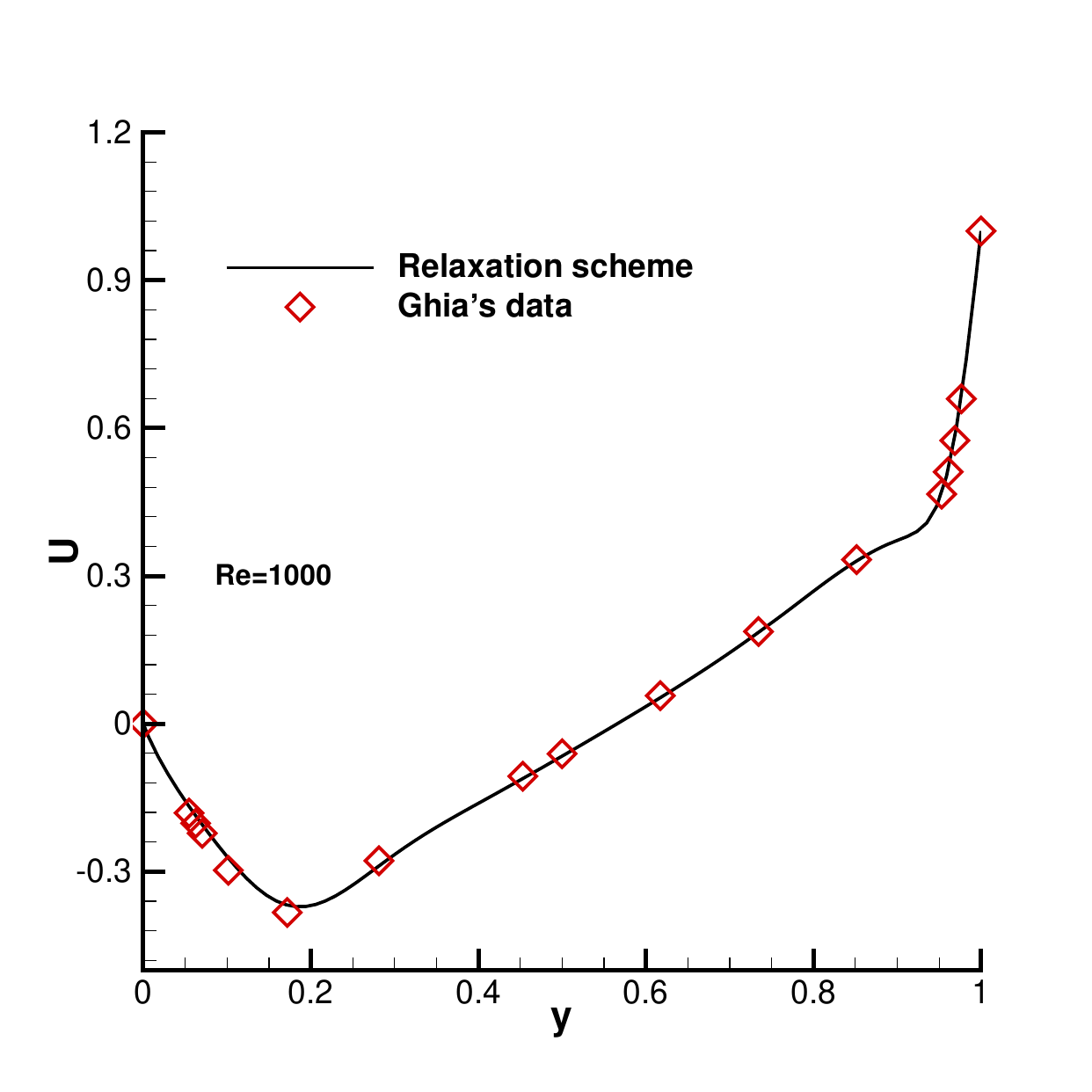}}
		\hspace{5mm}
		\subfigure{	\includegraphics[width=0.40\linewidth]{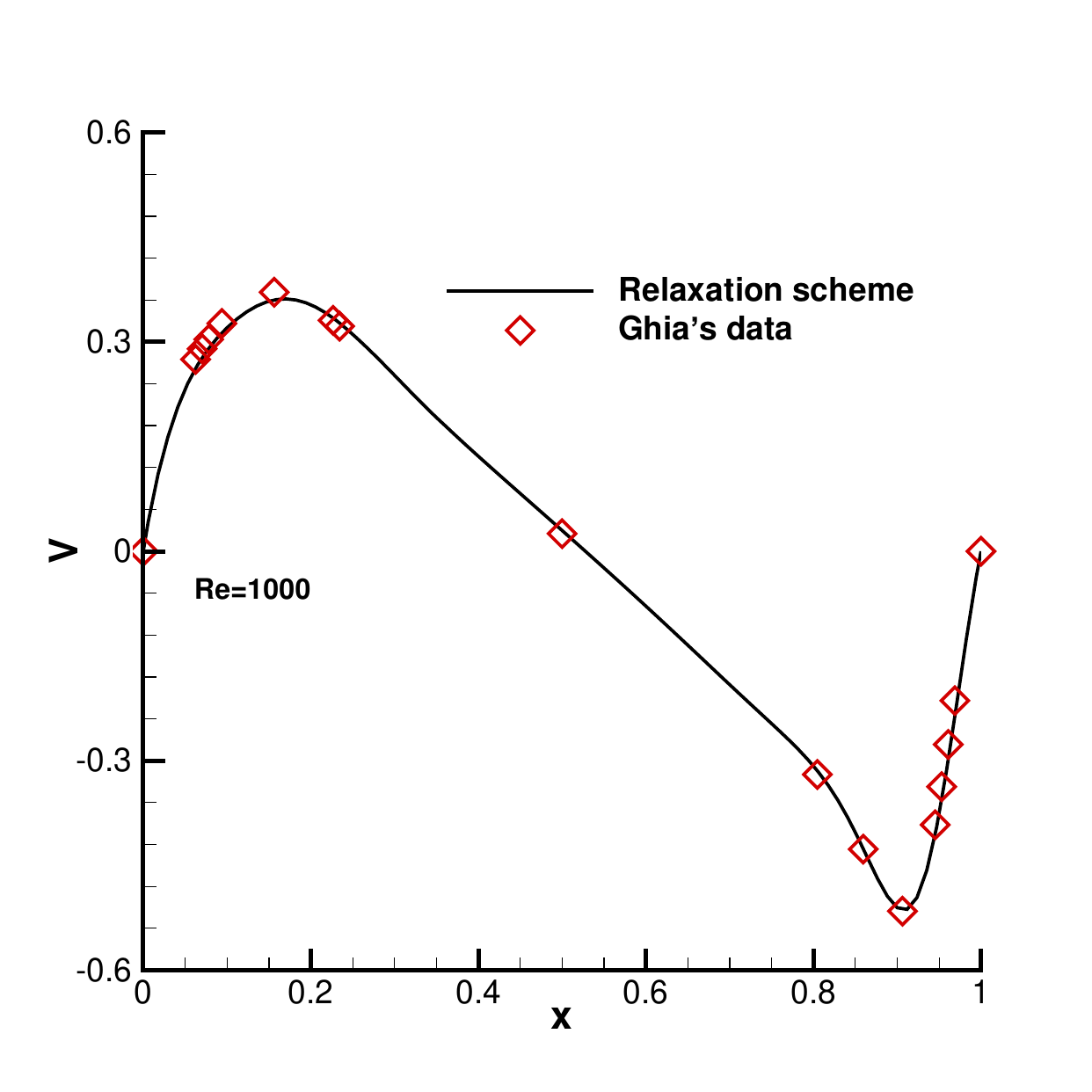}}
		\caption{\small Lid-driven cavity flow: $U$-velocity along vertical centerline line and $V$-velocity along horizontal centerline with $\Delta x=\Delta y=1/85$ for $Re=400$ and $1000$. The reference data are from Ghia et al.~\cite{Ghia1982}.}
		\label{fig:UV}
	\end{figure}
\end{test}
\begin{test}[Reflected shock-boundary layer interaction] \label{VSTP}
	In this test, we show the performance of the scheme proposed in \Cref{example5} for the viscous shock tube problem, which involves the interaction of reflected shock with boundary layer and has been studied extensively \cite{DT2009}. In this case, an ideal gas is at rest in a unit square $[0, 1]\times [0, 1]$. A membrane located at $x=0.5$ separates two different states of the gas and the dimensionless initial states are
	\begin{equation}
		(\rho, U, V, p) = \begin{cases}
			\begin{aligned}
				&(120, 0, 0, 120/\gamma ), \,\,\,\,\,\,\, &0\leq x\leq 0.5,\\
				&(1.2, 0, 0, 1.2/\gamma ), \,\,\,\,\,\,\, &0.5\leq x\leq 1.0.\\		
			\end{aligned}
		\end{cases}
	\end{equation}
	The reference velocity is based on the initial speed of sound, corresponding to a reference Mach number $M_0=1$. The membrane is removed at time zero and wave interaction occurs. A shock wave, followed by a contact discontinuity, moves right with a Mach number $Ma=2.37$ and reflects at the right end wall. After the reflection, it interacts with the contact discontinuity. The contact discontinuity and shock wave interact with the horizontal wall and create a thin boundary layer during their propagation. The solution will develop complex two-dimensional shock/shear/boundary-layer interactions, which requires not only strong robustness but also high resolution of a numerical scheme.
	
	The computational domain is set as $[0, 1] \times[0, 0.5]$. The symmetric boundary condition is applied on the upper boundary, and the non-slip and adiabatic boundary conditions are adopted on other boundaries. The Reynolds number $Re=200$ is chosen, which is based on a constant dynamic viscosity $\mu=0.005$. Other parameters are set as $\gamma=1.4$, $Pr=0.72$, $\epsilon=\mu/200$ and $a=0.7$. Numerical simulations are conducted for $400\times200$ and $600\times300$ uniform mesh points. The CFL number takes a value of $0.3$. The minmod limiter \eqref{minmod} with limiter parameter $\alpha=2.0$ is employed.
	
	\Cref{Tab:VSTP} presents the height of the primary vortex, achieving a good agreement with the reference data obtained by a high-order GKS (HGKS) method in \cite{ZXL2018} with the mesh size $\Delta x=\Delta y=1/1500$. \Cref{fig:VSTP} shows the density contours at $t=1.0$. The results match well with each other, and it can be observed that the complex flow structure are well resolved, including the lambda shock and the vortex structures. The density along the bottom wall is presented in \Cref{fig:VSTPden}, which match very well with the reference data. Though the relaxation scheme proposed in \Cref{example5} has only second-order accuracy, it resolves well such a complex flow, demonstrating the good performance of the current scheme in high-speed viscous flows.
	\begin{table}[htbp]
		\centering
		\setlength{\abovecaptionskip}{0.0cm}
		\setlength{\belowcaptionskip}{0.01cm}
		\setlength{\tabcolsep}{2.8mm}
		\begin{tabular}{cccc}
			\toprule
			Scheme   
			&\multicolumn{2}{c}{second-order relaxation scheme}   &  HGKS  \\
			\hline
			Mesh size         &$1/400$    & $1/600$   &$1/1500$  \\
			\hline
			Height   &0.160  & 0.164 &0.166 \\
			\bottomrule  
		\end{tabular}   
		\caption{\small Reflected shock-boundary layer interaction: comparison of the height of the primary vortex for $Re=200$. The reference data are from the HGKS method \cite{ZXL2018}.}   \label{Tab:VSTP}  
	\end{table}  
	\begin{figure}[htbp]
		\centering
		\setlength{\abovecaptionskip}{0.0cm}
		\setlength{\belowcaptionskip}{0.0cm}
		\subfigure{	\includegraphics[width=0.6\linewidth]{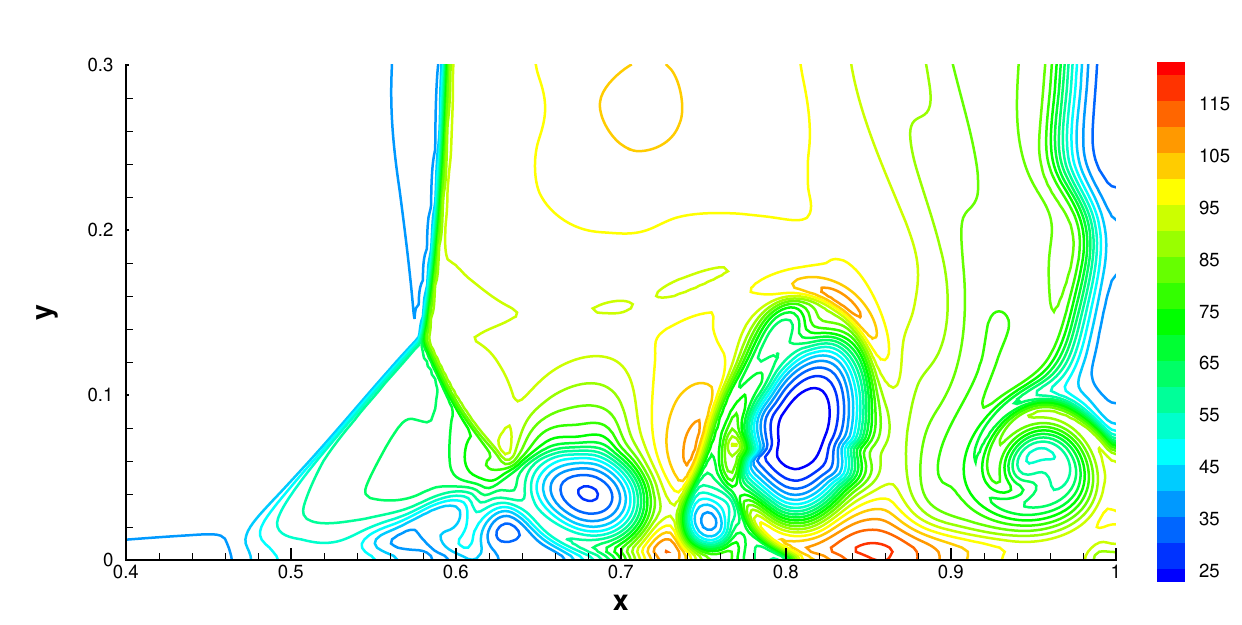}}
		\subfigure{	\includegraphics[width=0.6\linewidth]{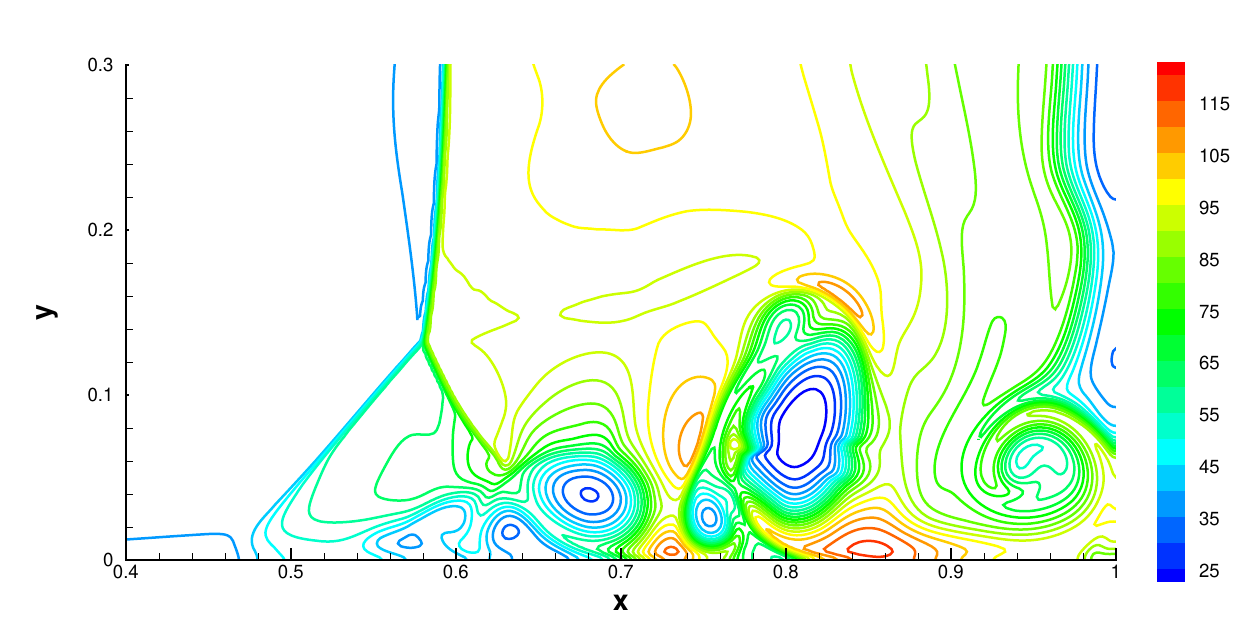}}
		\caption{\small Reflected shock-boundary layer interaction: density contours at $t=1.0$ with $\Delta x=\Delta y=1/400$ (top) and $\Delta x=\Delta y=1/600$ (bottom) for $Re=200$. 20 uniform contours from 25 to 120.}
		\label{fig:VSTP}
	\end{figure}
	\begin{figure}[htbp]
		\centering
		\setlength{\abovecaptionskip}{0.0cm}
		\setlength{\belowcaptionskip}{0.0cm}
		\includegraphics[width=0.5\linewidth]{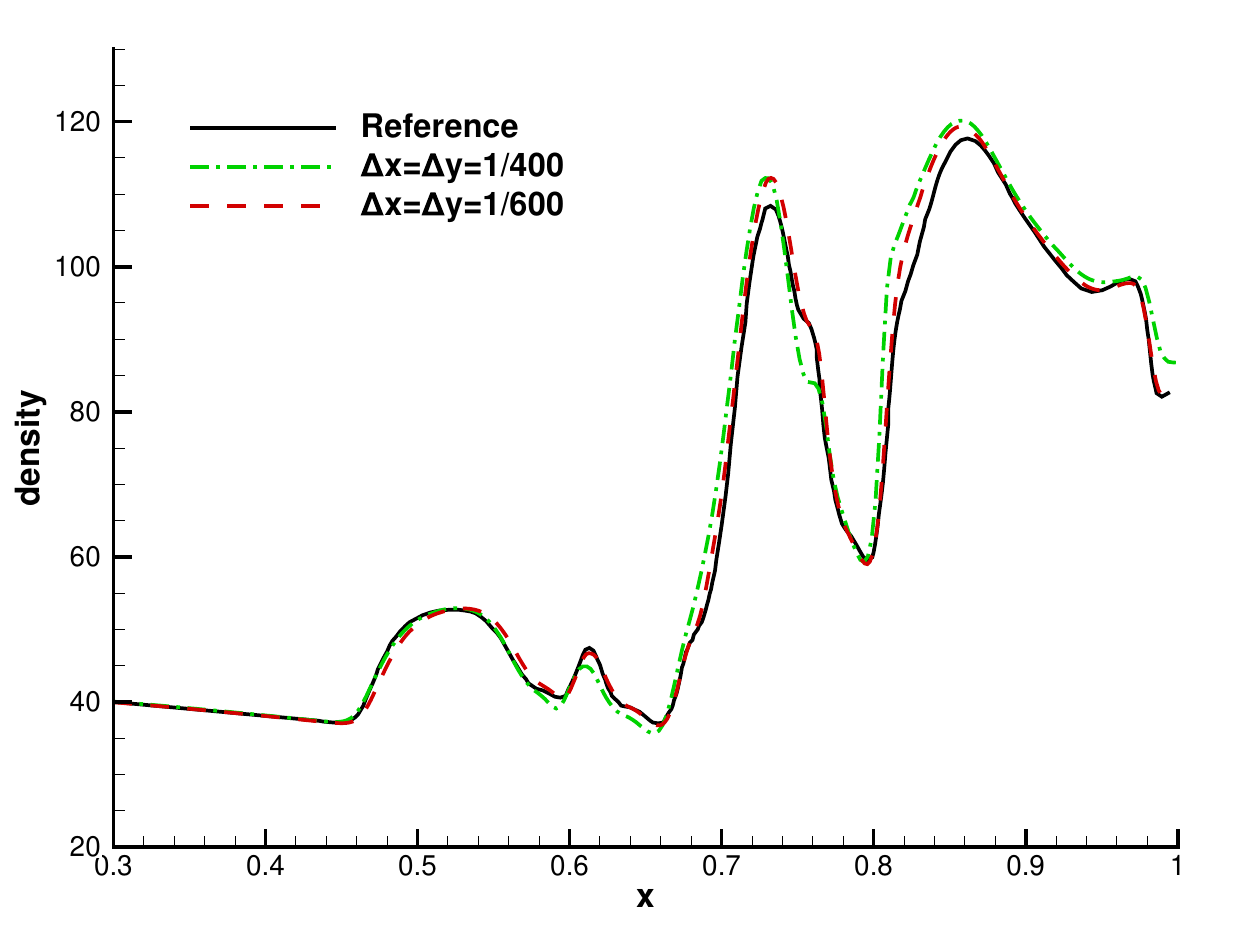}
		\caption{\small Reflected shock-boundary layer interaction: density along the bottom wall with different meshes for $Re=200$. The reference data are from the HGKS method \cite{ZXL2018}.}
		\label{fig:VSTPden}
	\end{figure}
\end{test}

\section{Discussions}
This paper proposes a relaxation model for general entropy dissipative system of viscous conservation laws via a flux relaxation approach. Our method ensures the hyperbolicity and the entropy dissipation property of the resulting system. The novelty of this paper lies in that our method provides an alternative way to design numerical schemes for general entropy dissipative system of viscous conservation laws such that the FV solvers developed for hyperbolic conservation laws can be employed as building blocks. For example, this paper develops second-order LW type relaxation schemes by using an IMEX-GRP method, see the scheme \eqref{exampleB0}--\eqref{exampleBend} for the 1-D viscous Burgers equation, the scheme in \Cref{example4} for the 1-D Navier-Stokes equations, and the scheme in \Cref{example5} for the 2-D Navier-Stokes equations. The above relaxation schemes are simple for implementation and enjoy the same advantages as the GRP method, such as the temporal-spatial coupling nature, relatively compact stencils and the multi-dimensionality. We point out that this approach can be applied to any entropy dissipative systems of viscous conservation laws, and can be extended to other formulations, e.g., the discontinuous Galerkin (DG) formulation. For relaxation schemes of 1-D scalar conservation laws, both the AP property and the dissipation property are shown.

Numerical examples are provided to validate the accuracy, robustness and high resolution of the current second-order relaxation schemes. Numerically, it is shown that with time step $\Delta t=O(\Delta x)$, the second-order accuracy can be achieved before the relaxation error of order $\ep$ becomes dominant. Moreover, the current second-order scheme exhibits good  computational performance for complex flow problems.

\appendix

\section{The 1-D Roe-type solver} \label{ROE}
This appendix includes the Roe-type solver for the Riemann problem 
\begin{equation}
	\begin{cases}
		\begin{aligned}
			&	U_t+F(U)_x = 0, \\
			&U(x,0)=\left.\begin{cases}
				U^l,\,\,\,\,\,\quad x<0,\\
				U^r,\,\,\,\,\,\quad x>0,
			\end{cases}	
			\right.
		\end{aligned}
	\end{cases}
\end{equation}
where $U$ and $F(U)$ are the same as in \eqref{CNS2} and \eqref{CNS3}, respectively, and $U^l$ and $U^r$ are the the limiting values of the initial data $U(x,0)$ on both sides of $(0,0)$. Let
\begin{equation}
	\tw^r=\tw(w^r),\,\,\,\tw^l=\tw(w^l),\,\,\,
	\br=\frac{\rho^l+\rho^r}{2},\,\,\,\bu=\frac{u^l+u^r}{2},
	\,\,\,\bE=\frac{E^l+E^r}{2},
\end{equation}
and define
\begin{equation}
	\widetilde{M}=\begin{pmatrix}
		0\,&  I_3\,\\
		\widetilde{B}^*P^{-1}\,& 0\,
	\end{pmatrix},\,\,\,
	\widetilde{B}^*=\begin{pmatrix}
		0\,&0\,&0\, \\
		0\,&\frac{4\mu}{3\epsilon}\,&0\,\\		
		0\,&\frac{4\mu}{3\epsilon}\bu\,&\frac{\kappa}{\epsilon}\,\\
	\end{pmatrix},\,\,\,
	P=\begin{pmatrix}
		1\,&0\,&0\, \\
		\bu\,&\br\,&0\,\\
		\bE\,&\br\bu\,&\frac{\br}{\gamma-1}\,\\
	\end{pmatrix}.
\end{equation}
Then, we can write
\begin{equation}
	F(U^r)-F(U^l)=\widetilde{M}(U^{r}-U^{l}),
\end{equation}
where $\widetilde{M}$ is indeed the so-called Roe matrix. Thus, a Roe-type solver as in \cite{Roe1981} takes
\begin{equation}
	\lim\limits_{t\rightarrow 0+}U(0,t)=\frac{1}{2}(U^r+U^l)-\frac{1}{2}\widetilde{R}
	\sign(\widetilde{\Lambda})(\widetilde{R})^{-1}(U^r-U^l),\label{1DRoe}
\end{equation}
where $\widetilde{R}$ is the matrix having the right eigenvectors of the Roe matrix $\widetilde{M}$ as its columns, and $\widetilde{\Lambda}$ is the diagonal matrix containing the eigenvalues.	

\section{Numerical treatment of boundary conditions} \label{boundarytreatment}
In the tests of the present paper, solid wall boundary conditions, such as the non-slip, isothermal and adiabatic boundary conditions, are treated by using the one-sided GRP (OS-GRP) solver \cite{LZ2022} as follows. With the 1-D problem \eqref{1Dsc2} as an example, suppose that $x=x_{-1/2}=0$ is a solid wall (boundary). For the Dirichlet boundary condition $u(x_{-1/2},t)=u_b(t)$, the boundary value $v(x_{-1/2},t_n)$ is evaluated by 
\begin{equation}
	\begin{aligned}
		&v^n_{-1/2}=f(u^n_{-1/2})-\mu\phi(u^n_{-1/2})\big(\frac{u^n_0-u^n_{-1}}{\Delta x}\big).\\
		&u^n_{-1}=2u^n_{-1/2}-u^n_0,\,\,\,\,\,u^n_{-1/2}=u_b(t_n).
	\end{aligned} 
\end{equation}
For the Neumann boundary condition $u_x(x_{-1/2},t)=0$, $v(x_{-1/2},t_n)$ is evaluated by
\begin{equation}
	v^n_{-1/2}=f(u^n_{-1/2}),\,\,\, \,\,\, u^n_{-1/2}=u^n_0.
\end{equation}
Then, once both of $u^n_{-1/2}$ and $v^n_{-1/2}$ are available, the numerical fluxes on the boundary can be obtained by using the acoustic version of the OS-GRP method, which showed significant benefits on avoiding spurious wave reflections at the computational boundaries \cite{LZ2022}. For the general multi-dimensional case, such a determination of the boundary conditions for the artificial variable $\bfv$ can be done componentwise. 

Other boundary conditions, such as the periodic, symmetric, in-flow and out-flow boundary conditions, are implemented with the help of the extrapolation technique.

\section*{Acknowledgments}
The first author would like to appreciate Bangwei She and Zhifang Du for their careful reading of the paper and detailed advice on improving the quality of the paper. The first author would also thank Liang Pan, Yue Wang, Xin Lei and Ang Li for their valuable suggestions during the discussions.

\end{document}